\documentclass[a4paper,10pt]{scrartcl}
\usepackage[utf8]{inputenc}
\usepackage[english]{babel}
\usepackage[fixlanguage]{babelbib}
\usepackage{cite}
\usepackage{amssymb, amsmath, amstext, amsopn, amsthm, amscd, amsxtra, amsfonts}

\usepackage{paralist}
\usepackage[T1]{fontenc}
\usepackage{ifthen}
\usepackage{colonequals}
\usepackage{mathtools} 

\usepackage[all]{xy}
\SelectTips{cm}{}
\usepackage{tikz}

\usepackage[bookmarks,bookmarksnumbered,plainpages=false]{hyperref}

\swapnumbers
\newtheoremstyle{standard}
 {16pt}  
 {16pt}  
 {}  
 {}  
 {\bfseries}
 {}  
 { } 
 {{\thmname{#1~}}{\thmnumber{#2.}}\thmnote{~(#3)}} 

\newtheoremstyle{kursiv}
 {16pt}  
 {16pt}  
 {\itshape}  
 {}  
 {\bfseries}
 {}  
 { } 
 {{\thmname{#1~}}{\thmnumber{#2.}}\thmnote{~(#3)}} 

\theoremstyle{standard}
\newtheorem{defn} [subsection]{Definition}
\newtheorem{ex} [subsection]{Example}
\newtheorem{rem}   [subsection]{Remark}
\newtheorem{nota}   [subsection]{Notation}
\newtheorem{setup} [subsection]{}

\theoremstyle{definition}

\theoremstyle{kursiv}
\newtheorem{thm}[subsection]{Theorem}
\newtheorem{prop} [subsection]{Proposition}
\newtheorem{cor} [subsection]{Corollary}
\newtheorem{lem} [subsection]{Lemma}

\setdefaultenum{\normalfont (a)}{i.}{A.}{1.}

\newcommand{\Evol}{\mathrm{Evol}}
\newcommand{\im}{\mathrm{im}}

\newcommand{\evol}{\mathrm{evol}}
\newcommand{\ev}{\mathrm{ev}}

\newcommand{\di}{\mathrm{d}}
\newcommand{\id}{\mathrm{id}}

\newcommand{\N}{\mathbb{N}}

\newcommand{\R}{\mathbb{R}}
\newcommand{\Q}{\mathbb{Q}}
\newcommand{\K}{\mathbb{K}}
\newcommand{\C}{\mathbb{C}}

\newcommand{\LB}[1][\cdot \hspace{1pt} , \cdot]{[\hspace{1pt} #1 \hspace{1pt} ]}

\DeclareMathOperator{\Fl}{Fl}
\newcommand{\set}[1]{\{  #1 \}}
\newcommand{\setm}[2]{\left\{\, #1 \,\middle\vert\, #2\,\right\}}
\newcommand{\norm}[1]{\left\lVert #1 \right\rVert}

\newcommand{\ve}{\varepsilon}
\newcommand{\coloneq}{\colonequals}

\DeclareMathOperator{\PR}{Pr}
\DeclareMathOperator{\OST}{OST}
\DeclareMathOperator{\splitting}{SP}
\newcommand{\SP}[1]{\splitting ( #1 )}

\DeclareMathOperator{\one}{\mathbf{1}}
\DeclareMathOperator{\an}{\mathbf{a}} 
\DeclareMathOperator{\bn}{\mathbf{b}} 	
\DeclareMathOperator{\cn}{\mathbf{c}} 
\DeclareMathOperator{\zero}{\mathbf{0}}
\DeclareMathOperator{\LOG}{Log}
\DeclareMathOperator{\BN}{Box} 
\newcommand{\BGp}{\ensuremath{G_{\mathrm{TM}}}}
\newcommand{\BGC}{\ensuremath{G_{\mathrm{TM}}^{\C}}}
\newcommand{\SymplTM}{\SymplTMR^\C}
\newcommand{\SymplTMR}{\ensuremath{S_{\mathrm{TM}}}}
\newcommand{\RT}{\ensuremath{\tT}}

\newcommand{\hH}{\ensuremath{\mathcal{H}}}

\newcommand{\tT}{\ensuremath{\mathcal{T}}}
\newcommand{\pP}{\ensuremath{\mathcal{P}}}

\newcommand{\Lf}{\ensuremath{\mathbf{L}}}
\newcommand{\dd}{\mathrm{d}}

\tikzstyle dtree=[grow'=up,sibling distance=2mm,level distance=2mm,thick]
\tikzstyle dtree node=[scale=0.3,shape=circle,very thin,draw]
\tikzstyle dtree black node=[style=dtree node,fill=black]
\newcommand{\onenode}{
  \begin{tikzpicture}[dtree]
    \node[dtree black node] {}
    ;
  \end{tikzpicture}
}

\newcommand{\twonode}{
\begin{tikzpicture}[dtree]
  \node[dtree black node] {}
  child { node[dtree black node] {} }
  ;
\end{tikzpicture}
}

\title{The Lie group structure of the Butcher group}
 \author{Geir Bogfjellmo\footnote{NTNU Trondheim \href{mailto:geir.bogfjellmo@math.ntnu.no}{geir.bogfjellmo@math.ntnu.no}, Phone: +47 73591753, Fax: +47 73593524\qquad(Corresponding author)}%
 \ \ and Alexander
Schmeding\footnote{NTNU Trondheim
\href{mailto:alexander.schmeding@math.ntnu.no}{alexander.schmeding@math.ntnu.no}
}}

\begin{document}

\maketitle

\begin{abstract}
 The Butcher group is a powerful tool to analyse integration methods for ordinary differential equations, in particular Runge--Kutta methods.
 In the present paper, we complement the algebraic treatment of the Butcher group with a natural infinite-dimensional Lie group structure. 
 This structure turns the Butcher group into a real analytic Baker--Campbell--Hausdorff Lie group modelled on a Fr\'{e}chet space.
 In addition, the Butcher group is a regular Lie group in the sense of Milnor and contains the subgroup of symplectic tree maps as a closed Lie subgroup.
 Finally, we also compute the Lie algebra of the Butcher group and discuss its relation to the Lie algebra associated to the Butcher group by Connes and Kreimer.  
\end{abstract}

\medskip


\medskip


\tableofcontents

\section*{Introduction and statement of results}
 In his seminal work \cite{Butcher72} J.C.\ Butcher introduced the Butcher group as a tool to study order conditions for a class of integration methods.
 Butcher's idea was to build a group structure for mappings on rooted trees. 
 The interplay between the combinatorial structure of rooted trees and this group structure enables one to handle formal power series solutions of non-linear ordinary differential equations.
 As a consequence, an efficient treatment of the algebraic order conditions arising in the study Runge--Kutta methods became feasible.
 Building on Butcher's ideas, numerical analysts have extensively studied algebraic properties of the Butcher group and several of its subgroups. 
 The reader is referred to \cite{CHV2010,HLW2006,Brouder-04-BIT} and most recently the classification of integration methods in \cite{MMMV14}.
 \bigskip
 
 The theory of the Butcher group developed in the literature is mainly algebraic in nature. 
 An infinite-dimensional Lie algebra is associated to it via a link to a certain Hopf algebra (cf.\ \cite{Brouder-04-BIT} and Remark \ref{remark: CKLie}), but is not derived from a differentiable structure.
 However, it is striking that the theory of the Butcher group often builds on an intuition which involves a differentiable structure.
 For example, in \cite[IX Remark 9.4]{HLW2006} a derivative of a mapping from the Butcher group into the real numbers is computed, and the authors describe a ``tangent space'' of a subgroup.
 In \cite{CMSS93}, the authors calculate the exponential map by solving ordinary differential equations. The calculation in loc.cit.\ can be interpreted as a solution of ordinary differential equations evolving on the Butcher group.
 Differentiating functions and solving differential equations both require an implicit assumption of a differentiable structure.
 The present paper aims to explicitly describe a differentiable structure on the Butcher group that corresponds to the structure implicitly used in \cite{HLW2006,CMSS93}.
 To the authors' knowledge this is the first attempt to rigorously construct and study the infinite-dimensional manifold structures on the Butcher group and connect it with the associated Lie algebra.\bigskip
 
 Guiding our approach is the idea that a natural differentiable structure on the Butcher group should reproduce on the algebraic side the well known formulae for derivatives and objects considered in numerical analysis. 
 To construct such a differentiable structure on the Butcher group we base our investigation on a concept of $C^r$-maps between 
 locally convex spaces.
 This calculus is known in the literature as (Michal--)Bastiani~\cite{bastiani64} or Keller's $C^r_c$-theory~\cite{keller1974} (see \cite{milnor1983,hg2002a,neeb2006} for streamlined expositions). 
 In the framework of this theory, we construct a differentiable structure which turns the Butcher group into a (locally convex) Lie group modelled on a Fr\'{e}chet space.
 Then the Lie theoretic properties of the Butcher group and the subgroup of symplectic tree maps are investigated.
 In particular, we compute the Lie algebras of these Lie groups and relate them to the Lie algebra associated to the Butcher group.  
 \bigskip
 
 We now go into some more detail and explain the main results of the present paper.
 Let us first recall some notation and the definition of the Butcher group. 
 Denote by $\RT$ the set of all rooted trees with a finite positive number of vertices (cf.\ Section \ref{sect: BG-LGP}). 
  Furthermore, we let $\emptyset$ be the empty tree. 
  Then the Butcher group is defined to be the set of all tree maps which map the empty tree to $1$, i.e.\
  \begin{displaymath}
   \BGp = \{a \colon \RT \cup \{ \emptyset \} \rightarrow \R \mid a (\emptyset) = 1\}.
  \end{displaymath}
  To define the group operation, one interprets the values of a tree map as coefficients of a (formal) power series. 
  Via this identification, the composition law for power series induces a group operation on $\BGp$.\footnote{The same construction can be performed for tree maps with values in the field of complex numbers. 
  The group $\BGC$ of complex valued tree maps obtained in this way will be an important tool in our investigation. In fact, $\BGC$ is a complex Lie group and the complexification (as a Lie group) of the Butcher group.}
  We refrain at this point from giving an explicit formula for the group operations and refer instead to Section \ref{sect: BG-LGP}.\bigskip
  
  Note that the Butcher group contains arbitrary tree maps, i.e.\ there is no restriction on the value a tree map can attain at a given tree.
  Thus the natural choice of model space for the Butcher group is the space $\R^{\RT}$ of all real-valued tree maps, with the topology of pointwise convergence.
  Observe that $\RT$ is a countable (infinite) set, whence $\R^\RT$ is a Fr\'{e}chet space, i.e.\ a locally convex space whose topology is generated by a complete translation invariant metric.
  Now the results in Section \ref{sect: BG-LGP} subsume the following theorem.
  \bigskip
  
  \textbf{Theorem A} \emph{The Butcher group $\BGp$ is a real analytic infinite dimensional Lie group modelled on the Fr\'{e}chet space $\R^{\RT}$. }
 \bigskip
 
  Note that the topology on the model space of $\BGp$ can be defined in several equivalent ways. 
  For example as an inverse (projective) limit topology.
  The Fr\'{e}chet space $\R^{\RT}$ can be described as the inverse limit $\varprojlim_{n\in \N} \R^{\RT_n}$ where $\RT_n$ is the (finite) set of rooted trees with at most $n$ roots.
  The projection $\PR_n:\R^{\RT} \to \R^{\RT_n}$ is obtained by restricting the tree mapping to trees with at most $n$ nodes. In numerical analysis, this corresponds to truncating a B-series by ignoring $\mathcal{O}(h^{n+1})$ terms.
  The topology on $\R^{\RT}$ is the coarsest topology such that $\PR_n$ is continuous for all $n$.
  This means that a sequence in $\R^{\RT}$ converges if and only if it converges in all projections.


  Furthermore, this topology is rather coarse, i.e.\ some subsets of $\BGp$ considered in applications in numerical analysis will \emph{not} be open with respect to this topology (see Remark \ref{rem: coarse:top}). 
  Nevertheless, we shall argue that the Lie group structure we constructed is the natural choice, i.e.\ it complements the intuition of numerical analysts and the known algebraic picture.
  
  To illustrate our point let us now turn to the Lie algebra associated to the Butcher group.
  For a tree $\tau$ we denote by $\SP{\tau}_1$ the set of non-trivial splittings of $\tau$, i.e.\ non-empty subsets $S$ of the nodes of $\tau$ such that the subgraphs $S_\tau$ (with set of nodes $S$) and $\tau \setminus \sigma$ are non-empty subtrees.
  With this notation we can describe the Lie bracket obtained in Section \ref{sect: BG-LAlg} as follows. 
  \bigskip
  
  \textbf{Theorem B} \emph{The Lie algebra $\Lf (\BGp)$ of the Butcher group is $(\R^\RT, \LB)$.
  Then the Lie bracket $\LB[\an , \bn]$ for $\an ,\bn \in \R^\RT$ is given for $\tau \in\RT$ by
   \begin{displaymath}
    \LB[\an,\bn](\tau) = \sum_{s \in \SP{\tau}_1}\left(\bn(s_{\tau}) \an (\tau \setminus s) - \bn(\tau \setminus s) \an (s_{\tau})\right).
   \end{displaymath}
  }
  
  At this point, we have to digress to put our results into a broader perspective.
  Working in renormalisation of quantum field theories, Connes and Kreimer have constructed in \cite{CK98} a Lie algebra associated to the Hopf algebra of rooted trees. 
  Later in \cite{Brouder-04-BIT} it was observed that the Butcher group can be interpreted as the character group of this Hopf algebra.
  In particular, one can view the Connes--Kreimer Lie algebra as a Lie algebra associated to the Butcher group. 
  As a vector space the Connes--Kreimer Lie algebra is the direct sum $\bigoplus_{n \in \N} \R = \{(a_n) \in \R^\N \mid a_n \neq 0 \text{ for only finitely many } n\}$ endowed with a certain Lie bracket.
  Now the Connes--Kreimer Lie algebra can canonically be identified with a subspace of $\Lf (\BGp)$ such that the Lie bracket $\LB$ from Theorem B restricts to the Lie bracket of the Connes--Kreimer Lie algebra.
  Thus we recover the Connes--Kreimer Lie algebra from our construction as a dense (topological) Lie subalgebra. 
  Our Lie algebra is thus the ``completion'' of the Connes--Kreimer Lie algebra discussed in purely algebraic terms in the numerical analysis literature (see Remark \ref{remark: CKLie}).
  The authors view this as evidence that the Lie group structure for the Butcher group constructed in this paper is the natural choice for such a structure.
  \bigskip
  
  We then investigate the Lie theoretic properties of the Butcher group. 
  To understand these results first recall the notion of regularity for Lie groups.
  
   Let $G$ be a Lie group modelled on a locally convex space, with identity element $\one$, and
  $r\in \N_0\cup\{\infty\}$. We use the tangent map of the right translation
  $\rho_g\colon G\to G$, $x\mapsto xg$ by $g\in G$ to define
  $v.g\coloneq T_{\one} \rho_g(v) \in T_g G$ for $v\in T_{\one} (G) =: \Lf(G)$.
  Following \cite{dahmen2011} and \cite{1208.0715v3}, $G$ is called
  \emph{$C^r$-regular} if for each $C^r$-curve
  $\gamma\colon [0,1]\rightarrow \Lf(G)$ the initial value problem
  \begin{displaymath}
   \begin{cases}
    \eta'(t)&= \gamma(t).\eta(t)\\ \eta(0) &= \one
   \end{cases}
  \end{displaymath}
  has a (necessarily unique) $C^{r+1}$-solution
  $\Evol (\gamma)\coloneq\eta\colon [0,1]\rightarrow G$, and the map
  \begin{displaymath}
   \evol \colon C^r([0,1],\Lf(G))\rightarrow G,\quad \gamma\mapsto \Evol
   (\gamma)(1)
  \end{displaymath}
  is smooth. If $G$ is $C^r$-regular and $r\leq s$, then $G$ is also
  $C^s$-regular. A $C^\infty$-regular Lie group $G$ is called \emph{regular}
  \emph{(in the sense of Milnor}) -- a property first defined in \cite{milnor1983}.
  Every finite dimensional Lie group is $C^0$-regular (cf. \cite{neeb2006}). Several
  important results in infinite-dimensional Lie theory are only available for
  regular Lie groups (see
  \cite{milnor1983,neeb2006,1208.0715v3}, cf.\ also \cite{KM97} and the references therein).
  Specifically, a regular Lie group possesses a smooth Lie group exponential map.
  \bigskip
  
  \textbf{Theorem C} \emph{ The Butcher group is
  \begin{enumerate}
   \item $C^0$-regular and thus in particular regular in the sense of Milnor,
   \item exponential and even a Baker--Campbell--Hausdorff Lie group, i.e.\ the Lie group exponential map $\exp_{\BGp} \colon \Lf (\BGp) \rightarrow \BGp$ is a real analytic diffeomorphism.
  \end{enumerate}
  }
  Finally, we consider in Section \ref{sect: subgp} the subgroup $\SymplTMR$ of symplectic tree maps studied in numerical analysis.
  The elements of $\SymplTMR$ correspond to integration methods which are symplectic for general Hamiltonian systems $y' = J^{-1} \nabla H(y)$ (cf.\ \cite[VI.]{HLW2006}).  
  
  Our aim is to prove that $\SymplTMR$ is a Lie subgroup of $\BGp$. 
  Using the algebraic characterisation of elements in $\SymplTMR$ it is easy to see that $\SymplTMR$ is a closed subgroup of $\BGp$.
  Recall from \cite[Remark IV.3.17]{neeb2006} that contrary to the situation for finite dimensional Lie groups, closed subgroups of infinite dimensional Lie groups need not be Lie subgroups.
  Nevertheless, our results subsume the following theorem.\bigskip
  
  \textbf{Theorem D} \emph{The subgroup $\SymplTMR$ of symplectic tree maps is a closed Lie subgroup of the Butcher group.
  Moreover, this structure turns the subgroup of symplectic tree maps into an exponential Baker--Campbell--Hausdorff Lie group.}\bigskip
  
   The characterisation of the Lie algebra of $\SymplTMR$ exactly reproduces the condition in \cite[IX. Remark 9.4]{HLW2006}, which characterises ``the tangent space at the identity of $\SymplTMR$''. 
   Note that in loc.cit.\ no differentiable structure on $\BGp$ or $\SymplTMR$ is considered and a priori it is not clear whether $\SymplTMR$ is actually a submanifold of $\BGp$.
   The differentiable structure of the Butcher group allows us to exactly recover the intuition of numerical analysts. 
   
  We have already mentioned that the Butcher group is connected to a certain Hopf algebra. 
  Using this connection, one can derive the constructions done here from the broader framework of Lie group structures for character groups of Hopf algebras developed in \cite{BDS2015}.
  In the present paper we avoid using the language of Hopf algebras and instead focus on concrete calculations. 
  Hence the reader need not be familiar with the overarching framework to understand the present paper.
  
\section{Preliminaries on the Butcher group and calculus}

In this section we recall some preliminary facts on the Butcher group and the differential calculus (on infinite-dimensional spaces) used throughout the paper. 
These results are well known in the literature but we state them for the readers convenience.
Finally, we will also discuss different natural topologies on the Butcher group and single out the topology which turns the Butcher group into a Lie group.
Let us first fix some notation used throughout the paper.

\begin{nota}
 We write $\N \coloneq \set{1,2,\ldots}$, respectively $\N_0 \coloneq \N \cup \set{0}$. As usual $\R$ and $\C$ denote the fields of real and complex numbers, respectively. 
 \end{nota}

\subsection*{The Butcher group}

We recommend \cite{HLW2006,CHV2010} for an overview of basic results and algebraic properties of the Butcher group. 

\begin{nota}
 \begin{enumerate}
  \item A \emph{rooted tree} is a connected \emph{finite} graph without cycles with a distinguished node called the \emph{root}. 
  We identify rooted trees if they are graph isomorphic via a root preserving isomorphism.
  
  Let $\RT$ be the set of all rooted trees with a finite number of vertices and denote by $\emptyset$ the empty tree. 
  We set $\RT_0 \coloneq \RT \cup \{\emptyset\}$.
  The \emph{order} $|\tau|$ of a tree $\tau \in \RT_0$ is its number of vertices. 
  \item An \emph{ordered subtree}\footnote{The term ``ordered'' refers to that the subtree remembers from which part of the tree it was cut.} of $\tau \in \RT_0$ is a subset $s$ of vertices of $\tau$ which satisfies 
    \begin{compactitem}
     \item[(i)] $s$ is connected by edges of the tree $\tau$,
     \item[(ii)] if $s$ is non-empty, it contains the root of $\tau$.
    \end{compactitem}
   The set of all ordered subtrees of $\tau$ is denoted by $\OST (\tau)$.
   Associated to an ordered subtree $s \in \OST (\tau)$ are the following objects:
    \begin{compactitem}
     \item A forest (collection of rooted trees) denoted as $\tau \setminus s$. 
     The forest $\tau \setminus s$ is obtained by removing the subtree $s$ together with its adjacent edges from $\tau$
     \item $s_\tau$, the rooted tree given by vertices of $s$ with root and edges induced by that of the tree $\tau$.
    \end{compactitem}
 \end{enumerate}
\end{nota}

\begin{setup}[Butcher group]
 Define the \emph{complex Butcher group} as the set of all tree maps 
  \begin{displaymath}
   \BGC = \setm{a \colon \RT_0 \rightarrow \C}{a(\emptyset) = 1}
  \end{displaymath}
 together with the group multiplication
 \begin{equation}\label{eq: BGP:mult}
  a\cdot b (\tau) \coloneq \sum_{s \in \OST (\tau)} b(s_\tau) a(\tau \setminus s) \quad \text{with} \quad a(\tau \setminus s) \coloneq \prod_{\theta \in \tau \setminus s} a (\theta).
 \end{equation}
 The identity element $e \in \BGp$ with respect to this group structure is  
  \begin{displaymath}
   e \colon \RT_0 \rightarrow \C, e (\emptyset) = 1 , \ e(\tau) = 0, \forall \tau \in \RT.
  \end{displaymath}
  We define the \emph{(real) Butcher group} as the real subgroup
  \begin{displaymath}
   \BGp = \setm{a \in \BGC}{\im \, a \subseteq \R}
  \end{displaymath}
  of $\BGC$. 
  Note that the real Butcher group is referred to in the literature as ``the Butcher group'', whence ``the Butcher group'' will always mean the real Butcher group.   
 \end{setup}
 
 
 \begin{rem}\label{rem: affsubs}
  For $\K \in \{ \R, \C \}$ the set of all tree maps $\K^{\RT_0} = \set{ a \colon \RT_0 \rightarrow \K }$ is a vector space with respect to the pointwise operations.
  Now the (complex) Butcher group coincides with the affine subspace $e + \K^{\RT}$, where $\K^{\RT}$ is identified with $\setm{a \in \K^{\RT_0}}{a (\emptyset) = 0}$.
 \end{rem}
 
 To state the formula for the inverse in the (complex) Butcher group we recall:
 
 \begin{nota}
 A \emph{partition} $p$ of a tree $\tau \in \RT_0$ is a subset of edges of the tree.
 We denote by $\pP (\tau)$ the set of all partitions of $\tau$ (including the empty partition). 
  Associated to a partition $p \in \pP (\tau)$ are the following objects 
    \begin{compactitem}
     \item A forest $\tau \setminus p$. 
     The forest $\tau \setminus p$ is defined as the forest that remains when the edges of $p$ are removed from the tree $\tau$,
     \item The \emph{skeleton} $p_\tau$, is the tree obtained by contracting each tree of $\tau \setminus p$ to a single vertex and by re-establishing the edges of $p$.
    \end{compactitem}
    
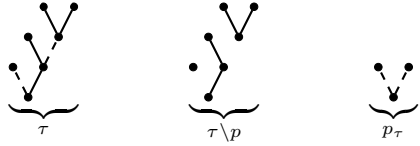
    \[\underbrace{\begin{tikzpicture}[dtree, scale=2]
      \node[dtree black node] {}
      child { node[dtree black node] {} edge from parent[dashed]
      }
      child { node[dtree black node] {}
      child { node[dtree black node] {} }
      child { node[dtree black node] {} edge from parent[dashed]
       child { node[dtree black node] {} edge from parent[solid]}
       child { node[dtree black node] {} edge from parent[solid]}
      }
      };
    \end{tikzpicture}}_\tau \qquad \qquad
    \underbrace{\begin{tikzpicture}[dtree, scale=2]
      \node[dtree black node] {}
      child { node[dtree black node] {} edge from parent[draw=none]
      }
      child { node[dtree black node] {}
      child { node[dtree black node] {} }
      child { node[dtree black node] {} edge from parent[draw=none]
       child { node[dtree black node] {} edge from parent[solid]}
       child { node[dtree black node] {} edge from parent[solid]}
      }
      };
      \end{tikzpicture}}_{\tau \setminus p} \qquad \qquad
    \underbrace{\begin{tikzpicture}[dtree, scale=2]
     \node[dtree black node] {}
      child { node[dtree black node] {} edge from parent[dashed]}
      child { node[dtree black node] {} edge from parent[dashed] };
    \end{tikzpicture}}_{p_\tau}
    \]
 Example of a partition $p$ of a tree  $\tau$, the forest $\tau \setminus p$ and the associated skeleton $p_\tau$. In the picture, the edges in $p$ are drawn dashed and roots are drawn at the bottom.
 \end{nota}
 
\begin{rem}[Inversion in the (complex) Butcher group]
  The inverse of an element in $\BGp$ can be computed as follows (cf.\ \cite{CHV2010})
  \begin{equation}\label{eq: BGP:invers}
   a^{-1} (\tau) = \sum_{p \in \pP (\tau)} (-1)^{|p_\tau|} a (\tau \setminus p) \quad \text{with }a(\tau \setminus p) = \prod_{\theta \in \tau \setminus p} a (\theta).
  \end{equation}
 \end{rem}
\subsection*{A primer to Locally convex differential calculus and manifolds}

 We will now recall basic facts on the differential calculus in infinite-dimensional spaces.
 The general setting for our calculus are locally convex spaces (see the extensive monographs \cite{MR0342978, jarchow1980}).  
 
 \begin{defn}
  Let $E$ be a vector space over $\K \in \{\R,\C \}$ together with a topology $T$.
  \begin{enumerate}
   \item $(E,T)$ is called \emph{topological vector space}, if the vector space operations are continuous with respect to $T$ and the metric topology on $\K$.
   \item A topological vector space $(E,T)$ is called \emph{locally convex space} if there is a family $\{p_i \colon E \rightarrow [0,\infty[ \mid i \in I\}$ of continuous seminorms for some index set $I$, such that 
     \begin{enumerate}
     \item the topology $T$ is the initial with respect to $\{p_i \colon E \rightarrow [0,\infty[\ \mid i \in I\}$, i.e.\ if $f \colon X \rightarrow E$ is a map from a topological space $X$, then $f$ is continuous if and only if $p_i \circ f$ is continuous for each $i\in I$, 
     \item if $x \in E$ with $p_i (x) = 0$ for all $i \in I$, then $x = 0$ (the semi-norms separate the points, i.e.\ $T$ has the Hausdorff property).
    \end{enumerate}
  In this case, the topology $T$ is \emph{generated by the family of seminorms} $\{p_i\}_{i \in I}$.
  Usually we suppress $T$ and write $(E, \{p_i\}_{i\in I})$ or simply $E$ instead of $(E,T)$.
  \item A locally convex space $E$ is called \emph{Fr\'{e}chet space}, if it is complete and its topology is generated by a countable family of seminorms.  
  \end{enumerate}
 \end{defn}
  
 \begin{rem}
 In a nutshell, a topological vector space carries a topology which is compatible with the vector space operations. 
 It turns out that the stronger conditions of a locally convex space yield an appropriate setting for infinite-dimensional calculus (i.e.\ many familiar results from calculus in finite dimensions carry over to these spaces).
 The spaces we are working with in the present paper will mostly be Fr\'{e}chet spaces. 
 Readers who are not familiar with these spaces should recall from \cite[Ch. I 6.1]{MR0342978} that the topology of a Fr\'{e}chet space is particularly nice, as it is induced by a metric. 
 
 Note that for a locally convex space $(E,\{p_i \mid i \in I\}$ the term \textquotedblleft locally convex\textquotedblright\ comes from the fact that the semi-norm balls 
  \begin{displaymath}
   B_r^{p_i} (x) = \{ y \in E \mid p_i (x-y) < r \} \text{ for } i\in I, r > 0 \text{ and } x \in E
  \end{displaymath}
 form convex neighbourhoods of the points. Since the topology is initial every open neighbourhood of a point contains such a convex neighbourhood. 
 \end{rem}

 \begin{ex}
  \begin{enumerate}
   \item Every normed space is a locally convex space (see \cite[Ch.\ I 6.2]{MR0342978}). 
   \item If $(E_\alpha , \{p_i^\alpha \mid i \in I_\alpha\})_{\alpha \in A}$ is a family of locally convex spaces, we denote by $E = \prod_{\alpha \in A} E_\alpha$ the cartesian product and let $\pi_\alpha \colon E \rightarrow E_\alpha$ be the projection onto the $\alpha$-component. Then $E$ is a locally convex space with the \emph{product topology} which is induced by the family of semi-norms $\{p_i^\alpha \circ \pi_\alpha \mid \alpha \in A, i \in I_\alpha \}$. 
   
   Note that with respect to the product topology each $\pi_\alpha$ is continuous and linear. 
   Furthermore, a mapping $f \colon F \rightarrow E$ from a locally convex space $F$ is continuous if and only if $\pi_\alpha \circ f$ is continuous. 
   If $A$ is countable and each $E_\alpha$ is a Fr\'{e}chet space, then $E$ is a Fr\'{e}chet space by \cite[Ch.\ I 6.2]{MR0342978} and \cite[Proposition 3.3.6]{jarchow1980}.
   \item Consider for $0< p < 1$ the $L^p$-spaces $L^p [0,1]$ of Lebesgue measurable functions on $[0,1]$.
   These spaces are complete topological vector spaces whose topology is induced by a metric, but they are not locally convex spaces (see \cite[Ch.\ I 6.1]{MR0342978}).
   \end{enumerate}\label{ex: lcvx:spac}
 \end{ex}
  
  As we are working beyond the realm of Banach spaces, the usual notion of Fr\'{e}chet differentiability can not be used\footnote{The basic problem is that the bounded linear operators do not admit a good topological structure if the spaces are not normable. In  particular, the chain rule will not hold for Fr\'{e}chet-differentiability in general for these spaces (cf.\ \cite[p. 73]{michor1980} or \cite{keller1974}).}  
  Moreover, there are several inequivalent notions of differentiability on locally convex spaces. 
  However, on Fr\'{e}chet spaces the most common choices coincide.
  For more information on our setting of differential calculus we refer the reader to \cite{hg2002a,keller1974}.
  The notion of differentiability we adopt is natural and quite simple, as the derivative is defined via directional derivatives.
  
\begin{defn}\label{defn: deriv} 
 Let $\K \in \set{\R,\C}$, $r \in \N \cup \set{\infty}$ and $E$, $F$ locally convex $\K$-vector spaces and $U \subseteq E$ open. 
 Moreover we let $f \colon U \rightarrow F$ be a map.
 If it exists, we define for $(x,h) \in U \times E$ the directional derivative 
 $$df(x,h) \coloneq D_h f(x) \coloneq \lim_{t\rightarrow 0} t^{-1} (f(x+th) -f(x)).$$ 
 We say that $f$ is $C^r_\K$ if the iterated directional derivatives
    \begin{displaymath}
     d^{(k)}f (x,y_1,\ldots , y_k) \coloneq (D_{y_k} D_{y_{k-1}} \cdots D_{y_1} f) (x)
    \end{displaymath}
 exist for all $k \in \N_0$ such that $k \leq r$, $x \in U$ and $y_1,\ldots , y_k \in E$ and define continuous maps $d^{(k)} f \colon U \times E^k \rightarrow F$. 
 If it is clear which $\K$ is meant, we simply write $C^r$ for $C^r_\K$.
 If $f$ is $C^\infty_\C$, we say that $f$ is \emph{holomorphic} and if $f$ is $C^\infty_\R$ we say that $f$ is \emph{smooth}.    
\end{defn}

\begin{ex}\label{ex: deriv: linear}
 Let $\lambda \colon E \rightarrow F$ be a continuous linear map between locally convex spaces, then for $x,y \in E$ we have
  \begin{displaymath}
   d\lambda (x,y) = \lim_{\K^\times \ni t \rightarrow 0} t^{-1} (\lambda (x+ty) - \lambda (x)) = \lambda (y).
  \end{displaymath}
 Hence we deduce that $d\lambda \colon E \times E \rightarrow F, (x,y) \rightarrow \lambda (y)$ is continuous and linear.
 In conclusion $\lambda$ is $C^1$ and its derivative is the map itself evaluated in the direction of derivation. Inductively, this implies that $\lambda$ is smooth. 
\end{ex}

On Fr\'{e}chet spaces our notion of differentiability coincides with the so called \textquotedblleft convenient setting \textquotedblright\ of global analysis outlined in \cite{KM97}.
Note that differentiable maps in our setting are continuous by default (which is in general not true in the convenient setting).
Later on a notion of analyticity for mappings between infinite-dimensional spaces is needed. 
Over the field of complex numbers we have the following assertion.

\begin{rem}
  \begin{enumerate}
   \item A map $f \colon U \rightarrow F$ is of class $C^\infty_\C$ if and only if it \emph{complex analytic} i.e.,
  if $f$ is continuous and locally given by a series of continuous homogeneous polynomials (cf.\ \cite[Proposition 1.1.16]{dahmen2011}).
  We then also say that $f$ is of class $C^\omega_\C$.
  \item If $f \colon U \rightarrow F$ is a $C^1_\C$-map and $F$ is complete, then $f$ is $C^\omega_\C$ by \cite[Remark 2.2]{hg2002a}.
  \end{enumerate}\label{rem: analytic}
\end{rem}

Now we discuss real analyticity for maps between infinite-dimensional spaces. 
Consider the one-dimensional case: 
A map $\R \rightarrow \R$ is real analytic if it extends to a complex analytic map $\C \supseteq U \rightarrow \C$ on an open $\R$-neighbourhood $U$ in $\C$.
We can proceed analogously for locally convex spaces by replacing $\C$ with a complexification.

\begin{setup}[Complexification of a locally convex space]
 Let $E$ be a real locally convex topological vector space. Endow $E_\C \coloneq E \times E$ with the following operation 
 \begin{displaymath}
  (x+iy).(u,v) \coloneq (xu-yv, xv+yu) \quad \text{ for } x,y \in \R, u,v \in E
 \end{displaymath}
 The complex vector space $E_\C$ with the product topology is called the \emph{complexification} of $E$. We identify $E$ with the closed real subspace $E\times \set{0}$ of $E_\C$. 
\end{setup}

\begin{defn}
 Let $E$, $F$ be real locally convex spaces and $f \colon U \rightarrow F$ defined on an open subset $U$. 
 We call $f$ \emph{real analytic} (or $C^\omega_\R$) if $f$ extends to a $C^\infty_\C$-map $\tilde{f}\colon \tilde{U} \rightarrow F_\C$ on an open neighbourhood $\tilde{U}$ of $U$ in the complexification $E_\C$.\footnote{If $E$ and $F$ are Fr\'{e}chet spaces, real analytic maps in the sense just defined coincide with maps which are continuous and can be locally developed into a power series. (see \cite[Proposition 4.1]{MR2402519})} 
\end{defn}

Now the important insight is that the calculus outlined admits a chain rule and many of the usual results of calculus carry over to our setting.
In particular, maps whose derivative vanishes are constant as a version of the fundamental theorem of calculus holds.
Moreover, the chain rule holds in the following form:

\begin{lem}[{Chain Rule {\cite[Propositions 1.12, 1.15, 2.7 and 2.9]{hg2002a}}}]
 Fix $k \in \N_0 \cup \{\infty , \omega\}$ and $\K \in \{\R , \C\}$ together with $C^k_\K$-maps $f \colon E \supseteq U \rightarrow F$ and $g \colon H \supseteq V \rightarrow E$ defined on open subsets of locally convex spaces.
 Assume that $g (U) \subseteq V$. Then $f\circ g$ is of class $C^{k}_\K$ and the first derivative of $f\circ g$ is given by 
  \begin{displaymath}
   d(f\circ g) (x;v) = df(g(x);dg(x,v)) \quad  \text{for all } x \in U ,\ v \in H
  \end{displaymath}
\end{lem}

The calculus developed so far extends easily to maps which are defined on non-open sets.
This situation occurs frequently if one wants to solve differential equations defined on closed intervals (one can generalise this even further see \cite{alas2012}).

\begin{defn}[Differentials on non-open sets]
  Let $[a,b] \subseteq \R$ be a closed interval wiht $a < b$. 
  A continuous mapping $f \colon [a,b] \rightarrow F$ is called $C^r$ if $f|_{]a,b[} \colon ]a,b[ \rightarrow F$ is $C^r$ and each of the maps\\ 
  $d^{(k)} (f|_{]a,b[}) \colon ]a,b[ \times E^k \rightarrow F$ admits a continuous extension $d^{(k)}f \colon [a,b] \times \R^k \rightarrow F$ (which is then unique).
  
  Let us agree on a special notation for differentials of maps on intervals: Define the map $\frac{\partial}{\partial t}f \colon [a,b]\rightarrow E, \frac{\partial}{\partial t}f(t) \coloneq df(t)(1)$.  
  If $f$ is a $C^r$-map, define recursively $\frac{\partial^k}{\partial t^k}f (t) \coloneq \frac{\partial}{\partial t}(\frac{\partial^{k-1}}{\partial t^{k-1}}f)(t)$
  for $k \in \N_0$ such that $k \leq r$.   
 \label{defn: non-open}
\end{defn}

\begin{ex}[Topologies on spaces of differentiable maps]\label{ex: funspace}
Consider a locally convex vector space $(E,\{p_i \mid i \in I\})$ over $\K \in \{\R,\C\}$ and a number $k \in \N_0 \cup \{\infty\}$. 
We define the vector space $C^k ([0,1] , E)$ of all $C^k_\K$-mappings $f \colon [0,1] \rightarrow E$ with the pointwise vector space operations.
This space is a locally convex space (over $\K$) with the \emph{topology of uniform convergence}, i.e.\ the topology generated by the family of semi-norms 
\begin{displaymath}
 \norm{f}_{i,r} \coloneq \sup_{t \in [0,1]} p_i \left(\frac{\partial^r}{\partial t^r} f (t)\right)  \quad \text{ for } i \in I \text{ and } 0\leq r \leq k
\end{displaymath}
The idea here is that the topology gives one control over the function and its derivatives. 
Note that if $E$ is a Banach space and $k < \infty$, these spaces are Banach spaces and for $k = \infty$ the space is a Fr\'{e}chet space (see \cite[4.3]{michor1980} and \cite[Remark 3.2]{MR1934608}) 
\end{ex}

Having the chain rule at our disposal we can define manifolds and related constructions which are modelled on locally convex spaces.

\begin{defn} Fix a Hausdorff topological space $M$ and a locally convex space $E$ over $\K \in \set{\R,\C}$. 
An ($E$-)manifold chart $(U_\kappa, \kappa)$ on $M$ is an open set $U_\kappa \subseteq M$ together with a homeomorphism $\kappa \colon U_\kappa \rightarrow V_\kappa \subseteq E$ onto an open subset of $E$. 
Two such charts are called $C^r$-compatible for $r \in \N_0 \cup \set{\infty,\omega}$ if the change of charts map $\nu^{-1} \circ \kappa \colon \kappa (U_\kappa \cap U_\nu) \rightarrow \nu (U_\kappa \cap U_\nu)$ is a $C^r$-diffeomorphism. 
A $C_\K^r$-atlas of $M$ is a family of pairwise $C^r$-compatible manifold charts, whose domains cover $M$. Two such $C^r$-atlases are equivalent if their union is again a $C^r$-atlas. 

A \emph{locally convex $C^r$-manifold} $M$ modelled on $E$ is a Hausdorff space $M$ with an equivalence class of $C^r$-atlases of ($E$-)manifold charts.
\end{defn}

 Direct products of locally convex manifolds, tangent spaces and tangent bundles as well as $C^r$-maps of manifolds may be defined as in the finite dimensional setting (see \cite[I.3]{neeb2006}). 
 The advantage of this construction is that we can now give a very simple answer to the question, what an infinite-dimensional Lie group is:
 
\begin{defn}
A \emph{Lie group} is a group $G$ equipped with a $C^\infty$-manifold structure modelled on a locally convex space, such that the group operations are smooth.
If the group operations are in addition ($\K$-)analytic we call $G$ a ($\K$)-\emph{analytic} Lie group. 
\end{defn}

\subsection*{Topologies on the Butcher group}

\begin{setup}\label{setup: affine}
 The function space $\K^{\RT_0}$ carries a natural topology, called the \emph{topology of pointwise convergence}. 
 This topology is given by the semi-norms 
  \begin{displaymath}
   p_\tau \colon \K^{\RT_0} \rightarrow \K , a \mapsto |a(\tau)|
  \end{displaymath}
 and turns $\K^{\RT_0}$ into a locally convex space. 
 To see this note that the  map \begin{displaymath}
          \Psi \colon \K^{\RT_0} \rightarrow \prod_{\tau \in \RT_0} \K , a \mapsto \an = (a(\tau))_{\tau \in \RT_0}
         \end{displaymath}
 is a vector space isomorphism which is also a homeomorphism if we endow the left hand side with the topology of pointwise convergence and the right hand side with the product topology.
 By abuse of language, we will refer to the topology of pointwise convergence on $\K^{\RT_0}$ also as the product topology.
 Since $\RT_0$ is countable, the space $\K^{\RT_0}$ is a Fr\'{e}chet space by Example \ref{ex: lcvx:spac} (b).
\end{setup}

\begin{setup}\label{rem: BGP:top}
 For the rest of this paper we canonically identify $\K^{\RT}$ as a subspace of $\K^{\RT_0}$ via 
  \begin{displaymath}
   \K^{\RT} \cong \{a \in \K^{\RT_0} \mid a(\emptyset) = 0\} = \ev_\emptyset^{-1} (0) \subseteq \K^{\RT_0}
  \end{displaymath}
 
 Now the (complex) Butcher group is the affine subspace $e+ \K^{\RT}$, where $e$ is the unit element in the (complex) Butcher group.
 Hence it is a manifold modelled on the locally convex space $\K^{\RT}$ with a global manifold chart given by the translation $-e$.
\end{setup}

As the Butcher group is an affine subspace of the locally convex space $\K^{\RT_0}$ (with the topology of pointwise convergence) we obtain the following useful facts.

\begin{lem}
 \begin{enumerate}
  \item For each $\tau \in \RT_0$ the mapping $\ev_\tau \colon \K^{\RT_0} \rightarrow \K , a \mapsto a(\tau)$ is a continuous linear map.
  \item $\K^{\RT}$ is a closed subspace of $\K^{\RT_0}$. 
  \item Let $M$ be an open subset of a locally convex space $F$ and consider a map $f\colon U \rightarrow e + \K^{\RT} \subseteq \K^{\RT_0}$ into the affine subspace (i.e.\ the (complex) Butcher group). 
 Then $f$ is of class $C^k_\K$ for $k \in \N_0 \cup \{\infty\}$ if and only if $\ev_\tau \circ f$ is of class $C^k_\K$ for all $\tau \in \RT$. 
 \end{enumerate}\label{lem: bgp:prod}
\end{lem}

\begin{proof} 
 \begin{enumerate}
  \item In \ref{setup: affine} we have seen that $\K^{\RT_0}$ is isomorphic (as a locally convex space) to the space $\prod_{\tau \in \RT_0} \K$ with the direct product topology. 
 Now each projection $\pi_\tau \colon \prod_{\tau \in \RT_0} \K \rightarrow \K$ is continuous and linear by Example \ref{ex: lcvx:spac} (b). 
 Clearly $\ev_\tau = \pi_\tau \circ \Psi$ (where $\Psi$ is the isomorphism of locally convex spaces from \ref{setup: affine}), whence (a) follows. 
  \item Note that $\K^{\RT} = \ev_\emptyset^{-1} (0)$ holds, whence it is closed in $\K^{\RT_0}$ since $\ev_\emptyset$ is continuous. 
  \item Using the manifold structure of the affine subspace and Remark \ref{rem: BGP:top}, identify $f$ with a mapping into $\K^{\RT_0}$ (with $\ev_\emptyset \circ f \equiv 1$).
 Now $f$ is of class $C^{k}_\K$ if and only if $\Psi \circ f$ is of class $C^{k}_\K$ (since $\Psi$ is a vector space isomorphism).
 This is the case if and only if $\pi_\tau \circ \Psi \circ f = \ev_\tau \circ f$ is of class $C^k_\K$ for each $\tau \in \RT_0$ (as a special case of \cite[Lemma 3.10]{alas2012}).
 However, from $\ev_\emptyset \circ f \equiv 1$ we deduce that $f$ is of class $C^k_\K$ if and only if $\ev_\tau \circ f$ is of class $C^k_\K$ for all $\tau \in \RT$.
 This proves the assertion.
 \end{enumerate}
\end{proof}

Since the (complex) Butcher group is an affine subspace of $\K^{\RT_0}$, tangent mappings are simply given by derivatives which can be computed directly in $\K^{\RT_0}$. 

 Consider a curve $c(t) \coloneq b + t\an$ which takes its image in $e + \C^{\RT}$ and satisfies $c(0) = b$ and $\left.\tfrac{\partial}{\partial t}\right|_{t=0} c(t) = \an$.
 By definition the tangent map $T_b f$ takes the derivative of a $C^1$-curve $c(t)$ in $\BGC$ with $c(0)=b$ and derivative $\left.\tfrac{\partial}{\partial t}\right|_{t=0} c(t) = \an$ to $\left.\tfrac{\partial}{\partial t}\right|_{t=0} f(c(t))$.

\begin{lem}\label{lem: tspace}
 The tangent space $T_b \BGC$ at a point $b \in \BGC$ coincides with $\C^{\RT}$.
 Moreover, the tangent map of a $C^1$-map $f \colon \BGC \rightarrow \BGC$ is given by the formula 
 \begin{displaymath}
  T_b f (\an) = \left.\frac{\partial}{\partial t}\right|_{t=0}f(b+t\an) \quad \text{for all } \an \in \C^{\RT} = T_b \BGC.
 \end{displaymath}
\end{lem}

\begin{lem}\label{lem: mspace:compl}
 The space $\C^{\RT_0}$ is the complexification of $\R^{\RT_0}$.
\end{lem}

\begin{proof}
 Taking identifications $\C^{\RT_0} \cong \prod_{\tau \in \RT_0} \C$ and $\R^{\RT_0} \cong \prod_{\tau \in \RT_0} \R$ the assertion follows from the definition of the product topology since $\C$ is the complexification of $\R$.
\end{proof}

\section{A natural Lie group structure for the Butcher group}\label{sect: BG-LGP}

In this section we construct a Fr\'{e}chet Lie group structure for the Butcher group. 
We will use the notation introduced in the previous section. 
Up to now, we have already obtained a topology on the (complex) Butcher group, which turns it into a complete metric space.
Moreover, this topology turns the (complex) Butcher group into an infinite-dimensional manifold modelled on the space $\K^{\RT}$.
We will now see that the group operations are smooth with respect to this structure, i.e.\ the group is an infinite-dimensional Lie group.

 \begin{thm}\label{thm: Butcher:Lie}
  The group $\BGC$ is a complex Fr\'{e}chet Lie group modelled on the space $\C^{\RT}$. 
  The complex Butcher group contains the Butcher group $\BGp$ as a real analytic Lie subgroup modelled on the Fr\'{e}chet space $\R^{\RT}$.
 \end{thm}
 
 \begin{proof} 
  We will first only consider the complex Butcher group $\BGC$ and prove that it is a complex Lie group.
  Recall from \ref{rem: BGP:top} that $\BGC$ is a manifold modelled on a Fr\'{e}chet space.
  Let us now prove that the group operations of the complex Butcher group are holomorphic (i.e.\ $C_\C^\infty$-maps) with respect to this manifold structure.\bigskip
 
 \textbf{Step 1: Multiplication in $\BGC$ is holomorphic.} 
 Consider the multiplication $m\colon \BGC \times \BGC \rightarrow \BGC$ on the affine subspace $\BGC$. 
 Recall from Lemma \ref{lem: bgp:prod} (b) that $m$ is holomorphic if and only if $\ev_{\tau} \circ m \colon \BGC \times \BGC \rightarrow \K$ is holomorphic for each $\tau \in \RT$.
 Hence we fix $\tau \in \RT$ and $a,b \in \BGC$ and obtain from \eqref{eq: BGP:mult} the formula
 \begin{align}
  \ev_\tau \circ m (a ,b) &= \sum_{s \in \OST (\tau)} b (s_\tau) \prod_{\theta \in \tau \setminus s} a (\theta). \label{eq: locmult}
 \end{align}
 We consider the summands in \eqref{eq: locmult} independently and fix $s \in \OST (\tau)$. 
 There are two cases for the evaluation $b(s_{\tau})$:
 
 If $s$ is empty then the map $b(s_\tau) = b(\emptyset) = \ev_\emptyset (b) \equiv 1$ is constant.
 Trivially in this case the map is holomorphic in $b$.
 
 Otherwise $s_{\tau}$ equals a rooted tree, whence $\ev_{s_{\tau}} \colon \C^{\RT_0} \rightarrow \C$ is continuous linear by Lemma \ref{lem: bgp:prod} and thus holomorphic in $b$ by Example \ref{ex: deriv: linear}. 
 Observe that the same analysis shows that each of the factors $\ev_\theta (a)$ are holomorphic in $a$ for all $\theta \in \tau \setminus s$. 
 Hence each summand in \eqref{eq: locmult} is a finite product of holomorphic maps and in conclusion multiplication in the group $\BGC$ is holomorphic.\bigskip
 
 \textbf{Step 2: Inversion in $\BGC$ is holomorphic.}
 Let $\iota \colon \BGC \rightarrow \BGC$ be the inversion in the complex Butcher group. 
 Again it suffices to prove that $\ev_\tau \circ \iota$ is holomorphic for each $\tau \in \RT$. 
 From \eqref{eq: BGP:invers} we derive for $\tau \in \RT_0$ and $a \in \BGC$ the formula 
  \begin{displaymath}
   \ev_\tau \circ \iota (a) = \sum_{p \in \pP (\tau)} (-1)^{|p_{\tau}|} a(\tau \setminus p) = \sum_{p \in \pP (\tau)} (-1)^{|p_{\tau}|} \prod_{\theta \in \tau \setminus p} a(\theta )
  \end{displaymath}
 Reasoning as in Step 1, we see that $\ev_\tau \circ \iota$ is holomorphic for each $\tau$ as a finite product of holomorphic maps, whence inversion in $\BGC$ is holomorphic.
 Summing up $\BGC$ is a complex Lie group modelled on the Fr\'{e}chet space $\C^{\RT}$.
 
 By Lemma \ref{lem: mspace:compl} the complexification $(\R^{\RT})_\C$ of the Fr\'{e}chet space $\R^{\RT}$ is the complex Fr\'{e}chet space $\C^{\RT}$. 
 We will from now on identify $\R^{\RT}$ with the real subspace $(\R \times \set{0})^{\RT}$ of $\C^{\RT} \subseteq \C^{\RT_0}$.
 By construction the Butcher group $\BGp$ is a real subgroup of $\BGC$ such that $(\R \times \{0\})^{\RT} \cap \BGC = \BGp$. 
 Moreover, the group operations of $\BGp$ extend in the complexification to the holomorphic operations of $\BGC$. 
 Thus the group operations of the Butcher group are real analytic, whence the Butcher group becomes a real analytic Lie group modelled on the Fr\'{e}chet space $\R^{\RT}$.
 \end{proof}

 Let us put the construction of the Lie group structure on the Butcher group into the perspective of applications in numerical analysis, by interpreting various statements from the literature in the product topology.
 
 \begin{rem}
 \begin{enumerate}
  \item In \cite{Butcher72}, Butcher states ``there is a sense in which elements of $\BGp$ can be approximated
 by elements of $G_0$'', where $G_0$ is the subset of $\BGp$ corresponding to Butcher's generalization of Runge--Kutta methods. The exact sense is stated in \cite[Theorem 6.9]{Butcher72}: 
 If $a\in  \BGp $ and $\RT_f$ is any finite subset of $\RT$, then there is $b \in  G_0$ such that $a|_{\RT_f} = b|_{\RT_f}$.
 
 Now \cite[Theorem 6.9]{Butcher72} implies that $G_0$ is dense in $\BGp$ with the product topology.
Indeed, let $a \in \BGp$ and $\RT_1 \subset \RT_2\subset \dotsb$ be an increasing sequence of subsets of $\RT$ with $\cup_{i=1}^{\infty} \RT_i = \RT$. 
Then loc.cit. implies there are elements $b_1, b_2, \dotsc$ in $G_0$ such that $b_i |_{\RT_i} = a |_{ \RT_i}$ for all $i$.
For any tree $\tau \in \RT$, we then have $b_i(\tau) = a(\tau)$ for all sufficiently large $i$ and $\lim_{i\rightarrow \infty} b_i = a$.
 \item In \cite[Equations (14) \& (15)]{CMSS93}, the authors arrive at differential equations for the coefficients $a_\lambda(\tau)$ for the flow of a modified vector field described by a B-series. 
With the differential structure on the Butcher group introduced in the present paper $a_\lambda$ itself can be described as a curve on $\BGp$ which solves an ordinary differential equation. 
The equation for the Lie group exponential \eqref{eq: ODE:exponential} (which we discuss in Section \ref{sect: LGP:exp}) is equivalent to the system in \cite{CMSS93}.
 \item Maybe the clearest example of use of topological/analytical intuition on the Butcher group in numerical literature is \cite[Remark 9.4]{HLW2006}.
Here the authors state that the ``coefficient mappings $b(\tau)$ [...] lie in the tangent space at $e(\tau)$ of the symplectic subgroup''.
In the present paper (cf.\ Section \ref{sect: subgp}), this statement is made precise, and the tangent space at $e$ of the symplectic subgroup is even the Lie algebra of the symplectic subgroup.
 \end{enumerate}
\end{rem}

 Originally the Butcher group was created by Butcher as a tool in the numerical analysis of Runge--Kutta methods. 
 In particular, Butcher's methods allow one to handle the combinatorial and algebraic difficulties arising in the analysis.
 How does the topology of the Lie group $\BGp$ figure into this picture? 

 The topology on the Lie groups $\BGp$ is the product topology of $\R^{\RT}$.
 From the point of view in numerical analysis, it would be desirable to have a finer topology on $\BGp$. 
 Let us describe a typical example where the product topology is too coarse: 

 \begin{rem}\label{rem: coarse:top}
 Consider an autonomous differential equation 
  \begin{displaymath}
   y' =f(y), \text{ with } f \colon U \rightarrow \R^n,\ n\in \N \text{ real analytic and } U \subseteq \R^n \text{ open}
  \end{displaymath}
 Now the idea is to apply numerical methods, obtain an approximate solution and compare approximate and exact solution.
 Recall that the elements in the Butcher group can be understood as coefficient vectors for numerical methods.
 In the product topology every neighbourhood of an element contains elements with infinitely many non-zero coefficients.
 To assure that the associated methods converges (with infinitely many non-zero coefficients) one has to impose growth restrictions onto \emph{infinitely many} coefficients (cf.\ \cite[Lemma 9]{HL1997}).
 The necessary conditions do not produce open sets in the product topology and we can not hope that the topology of the Lie group $\BGp$ will aid in a direct way in this construction in numerical analysis.
\end{rem}

 As the product topology is too coarse, can we refine it to obtain the necessary open sets? 
 A natural choice for a finer topology is the box topology which we define now.
 
 \begin{defn}[Box topology]
  Consider the sets 
    \begin{displaymath}
     \BN (x, \boldsymbol{\epsilon})\coloneq \setm{a\in \K^{\RT_0}}{|x(\tau) - a(\tau)| <  \boldsymbol{\epsilon}(\tau), \forall \tau \in \RT_0}
    \end{displaymath}
 where $ \boldsymbol{\epsilon}\colon \RT_0 \to ]0,\infty[$ and $x\in \K^{\RT_0}$.
 The sets $\BN (x, \boldsymbol{\epsilon})$ are called \emph{box neighbourhood} or \emph{box}.
 
 Now the set of all boxes $ \BN (x, \boldsymbol{\epsilon})$, where $(x, \boldsymbol{\epsilon})$ runs through $\K^{\RT_0} \times(]0,\infty[)^{\RT_0}$, forms a base of a topology on $\K^{\RT_0}$, called the \emph{box topology}.\footnote{i.e.\ every open set in the box topology can be written as a union of boxes.
 Note that we can not describe this topology via seminorms as it does \textbf{not} turn $\K^{\RT_0}$ into a topological vector space.} 
 \end{defn}
 
 \begin{rem}
 Note that the box topology allows one to control the growth of a function on trees in \textbf{all} trees at once, while the product topology gives only control over finitely many trees.
 Hence the usual growth restrictions from numerical analysis (on all coefficients of a B-series) lead to open (box) neighbourhoods.
  
 Moreover, the box topology is well behaved with respect to the direct sum of locally convex spaces:
 It is known (see \cite{CK98}) that one can associate a Lie algebra to the Butcher group which is given as a vector space by the direct sum 
 \begin{displaymath}
  \K^{(\RT_0)} \coloneq \{a \in \K^{\RT_0} \mid a(\tau)= 0 \text{ for almost all } \tau \in \RT_0\}.
 \end{displaymath}
 Now the natural locally convex topology on $\K^{(\RT_0)}$ is the box topology\footnote{Actually the natural topology is an inductive limit topology. However, as $\RT_0$ is countable this topology coincides with the box topology by \cite[Proposition 4.1.4]{jarchow1980}.}, i.e.\ the topology induced by the inclusion $\K^{(\RT_0)} \subseteq \K^{\RT_0}$ on the direct sum where $\K^{\RT_0}$ is endowed with the box topology. 
 Hence, these results indicate that actually one should consider the box topology.
 \end{rem} 
 
 However, the box topology on $\K^{\RT_0}$ is a very fine topology, i.e.\ it has very many open sets. 
 As a consequence it is especially hard to obtain maps $f \colon X \rightarrow \K^{\RT_0}$ which are continuous (in fact in Lemma \ref{lem: box:ntopgp} we will see that the group operations of the (complex) Butcher group are not continuous with respect to the box topology).
 
 Moreover, since the box topology has so many open sets, it turns $\K^{\RT_0}$ into a disconnected space. 
 By \cite[Theorem 5.1]{MR0160184} the connected component of $a \in \K^{\RT_0}$ is 
  \begin{displaymath}
   a+ \K^{(\RT_0)} \coloneq \{b \in \K^{\RT_0} \mid b(\tau) - a(\tau) =0 \text{ for almost all } \tau \in  \RT_0\} ,
  \end{displaymath}
 i.e.\ it is the direct sum with the box topology (which is a locally convex space, cf.\ \cite{MR0130553}).
 
 Since $\K^{\RT_0}$ with the box topology is a disconnected topological space, it fails to be a topological vector space.\footnote{While addition is continuous, scalar multiplication fails to be continuous, cf.\ the discussion of the problem in \cite{MR0130553}.}  
 Hence with the box topology on $\K^{\RT_0}$ one can not use techniques from calculus on $\K^{\RT_0}$ (since the standard notions of infinite-dimensional calculus require at least an ambient topological vector space).
 In addition, the box topology even fails to turn the (complex) Butcher group into a topological group.

 \begin{lem}\label{lem: box:ntopgp}
  If we endow $\BGC$ with the box topology, then the group operations become discontinuous. 
  Thus $\BGC$ can not be a topological group, whence it can not be a Lie group.
  A similar assertion holds for $\BGp$.
 \end{lem}
 
 \begin{proof}
  Let $e$ be the unit element of $\BGC$ and define $\boldsymbol{\epsilon} (\tau) \coloneq \tfrac{1}{|\tau|!}$ for $\tau \in \RT_0$. 
  Consider the box $\BN (e, \boldsymbol{\epsilon}) \subseteq \C^{\RT_0}$.
  Then $U \coloneq \BN (e, \boldsymbol{\epsilon}) \cap \BGC$ is an open neighbourhood of $e$ in $\BGC$.
  By construction $a \in U$ if and only if $|a (\tau)| < \frac{1}{|\tau|!}$ for all $\tau \in \RT$.
  We will prove now that there is no open $e$-neighbourhood $W \subseteq \BGC$ with $\iota (W) \subseteq U$, i.e.\ $\iota$ must be discontinuous at $e$.
  To see this we argue indirectly and assume that there is an open set $W$ with this property.
  Since $W$ is open, there is a box neighbourhood of $e$ contained in $W$. Thus we find $\ve > 0$ such that the map 
   \begin{displaymath}
    a_\ve \colon \RT_0 \rightarrow \C, a_\ve(\emptyset) = 1, a_\ve(\tau ) = \begin{cases}
                                                                     \ve & \text{if } \tau = \bullet \text{ (the one node tree)}\\
                                                                     0 & \text{else }
                                                                    \end{cases}
   \end{displaymath}
  is contained in $W$. Now by \eqref{eq: BGP:invers} we see that the inverse of $a_\ve$ satisfies 
   \begin{displaymath}
     a_\ve^{-1} (\tau) = \sum_{p \in \pP (\tau)} (-1)^{|p_\tau|} a_\ve(\tau \setminus p) = (-1)^{|\tau|-1} \underbrace{a_\ve(\bullet) a_\ve (\bullet) \cdots a_\ve(\bullet)}_{|\tau|\text{-times}} = (-1)^{|\tau|-1} \ve^{|\tau|}
   \end{displaymath}
  Hence $|\ev_\tau \circ \iota (a_\ve)| = \ve^{|\tau|}$ holds for all $\tau \in \RT_0$. 
  On the other hand $\iota (a_\ve) \in \iota (W) \subseteq U$ and thus we must have $\ve^{|\tau|} = |\ev_\tau \circ \iota (a)| < \tfrac{1}{|\tau|!}$ for all $\tau \in \RT_0$.
  We obtain a contradiction, whence $\iota$ can not be continuous in $e$.
  A similar argument shows that the multiplication can not be continuous in $(e,e)$.
  Thus $\BGC$ with the box topology can not be a topological group.
 \end{proof}


\section{The Lie algebra of the Butcher group}\label{sect: BG-LAlg}

In this section, the Lie algebra $\Lf(\BGC)$ of the complex Butcher group will be determined.
Note that the Lie bracket will be a continuous bilinear map on $\Lf (\BGC) = T_e \BGC$ (the tangent space at the identity) and thus $\Lf (\BGC)$ will be a topological Lie algebra.
The Lie bracket on $T_e \BGC$ is induced by the Lie bracket left invariant vector fields and we want to avoid computing their Lie bracket. 
This is possible by a classical argument by Milnor who computes the Lie algebra via the adjoint action of the group on its tangent space (see \cite[pp.\ 1035-1036]{milnor1983}).

To simplify the computation recall from Lemma \ref{lem: tspace} that the tangent space $T_e\BGC$ is simply the model space $\C^{\RT}$.
To distinguish elements in the model space from elements in $\BGC$ we will from now on always write $\an , \bn ,\cn , \ldots$ for elements in $\C^{\RT} \subseteq \C^{\RT_0}$. 
Pull back the multiplication (via the translation by $-e$) to a holomorphic map on the model space.
This map $\an \ast \bn \in \C^{\RT}$ is given by the formula
    \begin{displaymath}
     \an \ast \bn (\tau) = \sum_{s \in \OST (\tau)} (\bn +e) (s_\tau) (\an +e) (\tau \setminus s), \quad \text{ for }\tau \in \RT.
    \end{displaymath}
By construction, we derive for the zero-map $\zero \in \C^{\RT}$ the identities $\an \ast\, \zero = \an = \zero \ast \an$. 
Hence the constant term of the Taylor series of $\ast$ in $(\zero , \zero)$ (cf.\ \cite[Proposition 1.17]{hg2002a}) vanishes. 
Following \cite[Example II.1.8]{neeb2006}, the Taylor series is given as   
\begin{displaymath}
   \an \ast \bn =  \an + \bn + B( \an ,  \bn ) + \cdots.
  \end{displaymath}
Here $B( \an , \bn) = \left.\frac{\partial^2}{\partial r \partial t}\right|_{t,r=0} \hspace{-0.25cm} (t\an \ast r \bn)$ is a continuous $\C^{\RT}$-valued bilinear map and the dots stand for terms of higher degree. 
With arguments as in \cite[p.\ 1036]{milnor1983}, the adjoint action of $T_{e} \BGC =\C^{\RT} \subseteq \RT_0$ on itself is given by 
  \begin{displaymath}
   \text{ad} (\an) \bn = B( \an , \bn)  - B(  \bn, \an).
  \end{displaymath}
In other words, the skew-symmetric part of the bilinear map $B$ defines the adjoint action.

By \cite[Assertion 5.5]{milnor1983} (or \cite[Example II.3.9]{neeb2006}), the Lie bracket is given by $\LB[\an ,\bn ] = \text{ad}(\an ) \bn$.
To compute the bracket $\LB$, it is thus sufficient to compute the second derivative of $\ast$ in $(\zero ,\zero)$. 

\begin{setup}\label{setup: LBcomp}
 Fix $\an , \bn \in \C^{\RT}$ and compute $B( \an , \bn) = \left.\frac{\partial^2}{\partial r \partial t}\right|_{t,r=0} \hspace{-0.25cm} (t\an \ast r \bn)$. 
 Since $(t\an \ast r \bn)$ takes its values in $\C^{\RT}$, we can compute the derivatives componentwise, 
 i.e.\ for each $\tau \in \RT$ we have $\ev_\tau (B( \an , \bn)) = \left.\frac{\partial^2}{\partial r \partial t}\right|_{t,r=0} \hspace{-0.25cm} \ev_\tau  (t\an \ast r \bn)$.
 Example \ref{ex: deriv: linear} and the chain rule imply for any smooth map $f \colon \R \rightarrow \BGC$
    \begin{displaymath}
     \frac{\partial}{\partial r} (\ev_{\tau} \circ f) = d \ev_{\tau} \left(f; \frac{\partial}{\partial r} f\right) = \ev_{\tau} \circ \frac{\partial}{\partial r}f 
    \end{displaymath}
 since $\ev_\tau$ is continuous and linear for each $\tau \in \RT$.    
 Now fix $\tau \in \RT$ and use the formula \eqref{eq: locmult} to obtain a formula for the derivative. 
  \begin{align}
   \left.\frac{\partial^2}{\partial r \partial t}\right|_{t,r=0} \hspace{-0.25cm} \ev_\tau (t\an \ast r \bn) 
	  &=  \left.\frac{\partial^2}{\partial r \partial t}\right|_{t,r=0} \sum_{s \in \OST (\tau)} (r\bn +e)(s_{\tau}) \prod_{\theta \in \tau \setminus s} (t\an+e) (\theta) \notag \\
	  &= \sum_{s \in \OST (\tau)} \left.\frac{\partial}{\partial r}\right|_{r=0} \ev_{s_{\tau}} (r\bn+e) \left.\frac{\partial}{\partial t}\right|_{t=0} \prod_{\theta \in \tau \setminus r} (t\an+e) (\theta) \notag \\
	  &= \sum_{s \in \OST (\tau)} \ev_{s_{\tau}} \left(\left.\frac{\partial}{\partial r}\right|_{r=0} r\bn+e\right) \left.\frac{\partial}{\partial t}\right|_{t=0} \prod_{\theta \in \tau \setminus s} (t\an+e) (\theta) \notag\\
	  &= \sum_{s \in \OST (\tau), s \neq \emptyset} \ev_{s_\tau} (\bn) \left.\frac{\partial}{\partial t}\right|_{t=0} \prod_{\theta \in \tau \setminus s} (t\an+e) (\theta)\label{eq: preLB}
  \end{align}
  To compute the remaining derivative, we use the Leibniz-formula and pull the derivative into the argument of the $\ev_\theta$. 
  Hence the product in \eqref{eq: preLB} becomes: 
  \begin{displaymath}
   \left.\frac{\partial}{\partial t}\right|_{t=0} \prod_{\theta \in \tau \setminus s} \ev_\theta (t\an + e) 
			= \sum_{\theta \in \tau \setminus s} \ev_\theta (\an) \prod_{\gamma \in (\tau \setminus s) \setminus \{\theta\}} \ev_\gamma (\zero + e) (\theta).
  \end{displaymath}
 As $\zero +e = e$, each product (and thus each summand) such that there is a tree $\gamma \in (\tau \setminus s) \setminus \set{\theta , \emptyset}$ vanishes. 
 Hence, if the sum is non-zero it contains exactly one summand, i.e.\ $\tau \setminus s$ must be a tree. 
 Moreover, the derivative will only be non-zero if $\tau \setminus s$ is not the empty tree. 
 Otherwise,the first factor $\ev_\theta(\an)$ vanishes. 
 Before we insert these informations in \eqref{eq: preLB} to obtain a formula for $B(\an,\bn)$ let us fix some notation.
 \end{setup}
 
 \begin{nota}
  Let $\tau \in \RT_0$ be a rooted tree. 
  We define the \emph{set of all splittings} as 
    \begin{displaymath}
     \SP{\tau} \coloneq \setm{s \in \OST (\tau)}{\tau \setminus s \text{ consists of only one element}}
    \end{displaymath}
  Furthermore, define the set of \emph{non-trivial splittings}
  $\SP{\tau}_1 \coloneq \setm{\theta \in \SP{\tau}}{\theta \neq \emptyset, \tau}$.
  Observe that for each tree $\tau$ the order of trees in $\SP{\tau}_1$ is strictly less than $|\tau|$. 
  Thus for the tree $\bullet$ with exactly one node $\SP{\bullet}_1 = \emptyset$.
 \end{nota}

\begin{setup}\label{setup: Bkcomp}
 With the notation in place, we can finally insert the information obtained in \ref{setup: LBcomp} into \eqref{eq: preLB} to obtain the following formula for the $\tau$th component of $B(\an,\bn)$: 
  \begin{displaymath}
   \ev_\tau B(\an, \bn) = \sum_{s \in \SP{\tau}_1} \bn(s_{\tau}) \an (\tau \setminus s)
  \end{displaymath}
\end{setup}

\begin{thm}\label{thm: Butcher:LA}
  The Lie algebra of the complex Butcher group is $(\C^{\RT}, \LB)$, where the Lie bracket $\LB[\an , \bn]$ for $\an ,\bn \in \C^{\RT}$ is given for $\tau \in \RT$ by
  \begin{equation}\label{eq: LB}
   \LB[\an,\bn](\tau) = \sum_{s \in \SP{\tau}_1}\left(\bn(s_\tau) \an (\tau \setminus s) - \bn(\tau \setminus s) \an (s_\tau)\right).
  \end{equation}
  Note that by \eqref{eq: LB} $\LB$ restricts to a Lie bracket on $\Lf (\BGp) \cong(\R\times\set{0})^{\RT} \subseteq \Lf(\BGC)$.
 
 The Lie algebra of the Butcher group is $(\R^{\RT} , \LB)$, with the bracket induced by $\Lf(\BGC)$ on the subspace $\R^{\RT} \cong (\R\times\set{0})^{\RT}$. 
\end{thm}

\begin{proof}
 Clearly \eqref{eq: LB} follows directly from the computation in \ref{setup: LBcomp} and \ref{setup: Bkcomp} and the formula for the Lie bracket via the adjoint action on $T_e \BGC$. 
 
 From \eqref{eq: LB} we see that $\LB$ restricts to the subspace $(\R\times\set{0})^{\RT}$. 
 In Theorem \ref{thm: Butcher:Lie} we have seen that the Butcher group is contained as a real analytic subgroup of $\BGC$.
 In particular, we see that $T_e \BGp \cong (\R\times\set{0})^{\RT}$. 
 Clearly the calculation in \ref{setup: LBcomp} restrict to $(\R\times\set{0})^{\RT}$ and yield a Lie bracket for $\BGp$ on $T_e \BGp$.
\end{proof}

The Lie algebra of the Butcher group is not the only Lie algebra closely connected to the Butcher group. 
To explain this connection we briefly recall some classical results by Connes and Kreimer: 

\begin{rem}\label{remark: CKLie}
 In \cite{CK98} Connes and Kreimer consider a Hopf algebra $\hH$ of rooted trees.
 The algebra $\hH$ is the $\R$-algebra\footnote{In \cite{CK98} the authors work over the field $\Q$ of rational numbers. 
 However, by applying $\cdot \otimes_{\Q} \R$ to the $\Q$-algebras the same result holds for the field $\R$ (cf.\ \cite[p.\ 41]{CK98}). 
 The thesis of Mencattini, \cite{mencattini} contains an explicit computation for $\R$ and $\C$.}
 of polynomials on $\RT_0$ together with the coproduct 
 \begin{displaymath}
  \Delta \colon \hH \rightarrow \hH \otimes \hH , \tau \mapsto \sum_{s \in \OST (\tau)} (\tau \setminus s) \otimes s_\tau
 \end{displaymath}
and the antipode $S (\tau) \coloneq \sum_{p \in \pP (\tau)} (-1)^{|p_\tau|} (\tau \setminus p)$. 

As observed by Brouder (cf.\ \cite{Brouder-04-BIT}), the coproduct and antipode are closely related to the product and inversion in the Butcher group.
Indeed the Butcher group corresponds to the group of $\R$-valued characters of the Hopf algebra $\hH$ (see \cite[5.1]{CHV2010}).

In \cite[Theorem 3]{CK98} Connes and Kreimer constructed a Lie algebra $L_{CK}$ such that $\hH$ is the dual of the universal enveloping algebra of $L_{CK}$. 
From \cite[Eq. (99)]{CK98} we deduce that the Lie algebra $L_{CK}$ is given by the vector space $\bigoplus_{\tau \in \RT} \R$ with a suitable Lie bracket $\beta$.
Identifying $\bigoplus_{\tau \in \RT_0} \K \subseteq \prod_{\tau \in \RT_0} \K \rightarrow \K^{\RT_0}$ (which is the Lie algebra of the (complex) Butcher group) as in \ref{setup: affine}, the Lie bracket $\beta$ coincides with $\LB$ from Theorem \ref{thm: Butcher:LA} on the image of $\bigoplus_{\tau \in \RT_0} \K$.  
We conclude that the Lie algebra $\Lf(\BGp)$ of the Butcher group contains the Connes-Kreimer Lie algebra $L_{CK}$ as a subalgebra. 
Moreover, the subalgebra $L_{CK}$ is a dense subset of $\Lf(\BGp)$.
Thus we can identify the Lie algebra of the Butcher group with the completion of $L_{CK}$ as a topological vector space.
Note that this is the precise meaning of the term 'natural topological completion' used in \cite[p. 4]{LT07}. 
In loc.cit.\ the Connes-Kreimer algebra is discussed only in algebraic terms, whence the above observations put the remark into the proper topological context.
\end{rem}

Let us record a useful consequence of the computations in this section. 
Arguing as in \eqref{eq: preLB} we can derive a formula for the tangent mapping of the right-translation.

\begin{lem}\label{lem: deriv:trans}
 Fix $b \in \BGC$ and denote by $\rho_{b} \colon \BGC \to \BGC , x \mapsto x \cdot b$ the right translation. 
 For any $\an \in \C^{\RT} = T_e \BGC$ and $\tau \in \RT$, we then obtain the formula 
 \begin{equation}\label{eq: d:trans}
 \begin{aligned}
  \ev_\tau T_e \rho_{b} (\an) & =\an(\tau) + \sum_{s \in \SP{\tau}_1} b (s_{\tau}) \an (\tau \setminus s).
  \end{aligned}
 \end{equation}
\end{lem}

\begin{proof}
 Computing as in \eqref{eq: preLB} we obtain with Lemma \ref{lem: tspace} the desired formula 
  \begin{align*}
  \ev_\tau T_e \rho_b (\an)&= \left.\frac{\partial}{\partial t} \right|_{t=0} (b + t\an) \cdot b (\tau) = \left.\frac{\partial}{\partial t} \right|_{t=0} \sum_{s \in \OST(\tau)} b(s_\tau) (b+t\an)(\tau \setminus s) \\
		  &= \an (\tau) + \sum_{s \in \SP{\tau}_1} b(s_{\tau}) \an (\tau \setminus s). \qedhere
  \end{align*}
\end{proof}

\section{Regularity properties of the Butcher group}\label{sect: regularity}

 Finally we discuss regularity properties of the Lie group $\BGC$ and the Butcher group.
 Since we also want to establish regularity properties for the complex Lie group $\BGC$ several comments are needed: 
 Recall that holomorphicmaps are smooth with respect to the underlying real 
 structure\footnote{This follows from  \cite[Remark 2.12 and Lemma 2.5]{hg2002a} for manifolds modelled on Fr\'{e}chet spaces.},
 whence $\BGC$ carries the structure of a real Lie group.
 Now the complex Lie group $\BGC$ is called regular, if the underlying real Lie group is regular. 
 Thus for this section we fix the following convention. 
 \medskip

 \textit{Unless stated explicitly otherwise, all complex vector spaces in this section are to be understood as the underlying real locally convex vector spaces. 
  Moreover, differentiability of maps is understood to be differentiability with respect to the field $\R$.}

\begin{setup}\label{setup: diffeq}
Define the mapping 
  \begin{displaymath}
   f \colon [0,1] \times \BGC \times C^0 ([0,1] , \Lf (\BGC)) \rightarrow \C^{\RT_0}, (t,a, \eta) \mapsto T_e \rho_{a} (\eta(t)). 
   \end{displaymath}
 Recall that $f$ describes the right-hand side of the differential equation for regularity of $\BGC$ (where we have again identified $\ BGC$ with an affine subspace of $\C^{\RT_0}$).
  Moreover, \eqref{eq: d:trans} yields the formula 
  \begin{displaymath}
    f(t,a,\eta) (\tau) = \eta (t)(\tau) + \sum_{s \in \SP{\tau}_1} a (s_{\tau}) \eta (t) (\tau \setminus s)
  \end{displaymath}
\end{setup}

Let us first solve the differential equations for regularity with fixed parameters.

\begin{prop}\label{prop: solution}
Fix a continuous curve $\an \colon [0,1] \to \Lf (\BGC)$ and let $f$ be defined as in \ref{setup: diffeq}.
Then the differential equation 
  \begin{equation}\label{eq: ODE:regular}
  \begin{cases}
   \gamma'(t)&= T_e \rho_{\gamma (t)} (\an (t)) = f(t, \gamma (t) , \an)\\ 
   \gamma(0) &= e
  \end{cases}
 \end{equation}
 on $\BGC$ admits a unique solution on $[0,1]$.
\end{prop}

\begin{proof}
 From \ref{setup: diffeq} we deduce that the first line of \eqref{eq: ODE:regular} can be rewritten for a tree $\tau$ as
 \begin{equation} \label{eq: reg:ode}
 \gamma'(t) (\tau) = \ev_{\tau}(\gamma' (t)) = \ev_{\tau} (\an (t)) +  \sum_{s \in \SP{\tau}_1} \ev_{s_{\tau}}(\gamma (t)) \an (t)(\tau \setminus s)
 \end{equation}
 As $\ev_{s_{\tau}}$ is continuous and linear for each $s \in \SP{\tau}_1$, each summand in the sum in \eqref{eq: reg:ode} is linear in $\gamma$.
 Note that for any fixed rooted tree $\tau$, the number of nodes for trees in $\SP{\tau}_1$ is strictly less than $|\tau|$. 
 Choose an enumeration $\tau_1, \tau_2, \dotsc$ of rooted trees which respects the number of nodes grading of the trees i.e.\ the enumeration satisfies:
 \begin{equation}
 \label{eq: ordering}
   \text{For all } l,k \in \N \text{ with } |\tau_k| < |\tau_l| \text{ we have } k < l.
 \end{equation}
 Using the enumeration, we rewrite the right hand side of \eqref{eq: reg:ode} as 
 \begin{equation}\label{eq: lower:diag}
  \ev_{\tau_k} (\gamma' (t)) = \left(\sum_{l < k} A_{kl} (t,\an) \ev_{\tau_l} (\gamma (t))\right) + \ev_{\tau_k} (\an (t)), \quad \quad k \in \N.
 \end{equation}
 Here the coefficient $A_{kl} (t,\an)$ for $l < k$ is a finite (possibly empty) sum of of terms of the form $\an (t)(\tau_k \setminus s)$ with $s_{\tau_k} = \tau_l$. 
 Since $\an \colon [0,1] \rightarrow \Lf (\BGC) \subseteq \C^{\RT_0}$ is continuous, we see that the $A_{kl}$ depend continuously on $t$.
  Following \cite[\S 6]{deimling77}, we interpret the differential equations \eqref{eq: reg:ode} as a system of differential equations.
 From \eqref{eq: lower:diag}, we deduce that this system is strictly lower diagonal, i.e.\ the right hand side of the $j$th component depends only on the first $j-1$ variables. 
 Furthermore, it is an inhomogenous linear system.
 The differential equation can be solved by adapting the argument in \cite[p. 79-80]{deimling77} as follows:
 
 Lower diagonal systems can be solved iteratively component by component, if each solution exists on a time intervall $[0,\ve]$ for some fixed $\ve >0$.
 The system is nonhomogenous linear, the solution at each iteration is unique and exists for all times $t \in [0,1]$.
 Therefore, the equation \eqref{eq: ODE:regular} admits a unique global solution which can be computed iteratively (more details on this are given in Remark \ref{rem: ctbl:sys} below).\qedhere

%
\end{proof}

\begin{defn}
 By Proposition \ref{prop: solution} we can define the flow-map associated to \eqref{eq: ODE:regular} via
  \begin{displaymath}
   \Fl^f \colon [0,1] \times C^0 ([0,1],\Lf (\BGC)) \rightarrow \BGC, (t,\an) \mapsto \gamma_{\an} (t)
  \end{displaymath}
 where $\gamma_{\an}$ is the unique solution to \eqref{eq: ODE:regular}. 
\end{defn}

 To prove regularity of the Butcher group, we will show that the flow-map $\Fl^f$ satisfies suitable differentiability properties.   
 Let us review the construction of $\gamma_{\an}$.

\begin{rem}\label{rem: ctbl:sys}
 Consider $f$ as in \ref{setup: diffeq} and fix $\an \in C^0([0,1],\Lf (\BGC))$. 
 Furthermore, choose an enumeration of $\RT$ which satisfies \eqref{eq: ordering}.
 
 We define 
 $g_{k,\an} \colon [0,1] \times \C^{k} \rightarrow \C , g_{k,\an} (t,x) = \ev_{\tau_k} \circ f(t,\hat{x},\an)$ for $k \in \N$,
 where $\hat{x} \in \C^{\RT_0}$ satisfies $\ev_{\tau_i}(\hat{x})= x_i$ for $i\le k$.
 Then \eqref{eq: lower:diag} shows that the functions $g_{k,\an}$ are well-defined and continuous.
 Fix $n\in \N$. The system of linear (inhomogeneous) initial value problems
  \begin{equation}\label{eq: init:n}
   x_i' = g_{i,\an} (t,x_1,\ldots, x_i) ,\ x_i (0) = 0 , \ i \leq n
  \end{equation}
 admits a solution on $[0,1]$.
 Recall from \cite[p. 78]{deimling77} the following facts on its solution. 
 If $(x_1^n ,\ldots, x_n^n)$ is a solution to \eqref{eq: init:n}, then we may solve 
  \begin{displaymath}
   y' = g_{n+1,\an} (t,x_1^n,x_2^n, \ldots, x_n^n,y) , \ y(0)=0
  \end{displaymath}
 on $[0,1]$. In particular, the map $(x_1^n,\ldots , x_n^n,y)$ solves the system \eqref{eq: init:n} for $n+1$.
 Continuing inductively, we obtain $\gamma_{\an}$ as the solution of \eqref{eq: ODE:regular}.  
\end{rem}

 The fact that solutions to \eqref{eq: ODE:regular} can be found inductively by solving finite-dimensional ODEs allows us to  discuss differentiability properties of the flow-map. 
 To this end we need a technical tool, the calculus of $C^{r,s}$-mappings, which we recall now from \cite{alas2012}.

\begin{defn}
 Let $H_1$, $H_2$ and $F$ be locally convex spaces, $U$ and $V$ open subsets of
 $H_1$ and $H_2$, respectively, and $r,s \in \N_0 \cup \{\infty\}$.
 \begin{enumerate}
  \item A mapping
 $f\colon U \times V \rightarrow F$ is called a $C^{r,s}$-map if for all
 $i,j \in \N_0$ such that $i \leq r, j \leq s$, the iterated directional
 derivative
 \begin{displaymath}
  d^{(i,j)}f(x,y,w_1,\dots,w_i,v_1,\dots,v_j) \coloneq (D_{(w_i,0)} \cdots
  D_{(w_1,0)}D_{(0,v_j)} \cdots D_{(0,v_1)}f ) (x,y)
 \end{displaymath}
 exists for all
 $ x \in U, y \in V, w_1, \ldots , w_i \in H_1,  v_1, \ldots ,v_j \in H_2$ and
 yields continuous maps
 \begin{align*} 
  d^{(i,j)}f\colon    U \times V \times H^i_1 \times H^j_2 &\rightarrow F,\\ 
  (x,y,w_1,\dots,w_i,v_1,\dots,v_j)&\mapsto (D_{(w_i,0)} \cdots D_{(w_1,0)}D_{(0,v_j)} \cdots D_{(0,v_1)}f ) (x,y).
 \end{align*}
 \item In (a) all spaces $H_1,H_2$ and $F$ were assumed to be modelled over the same $\K \in \{\R,\C\}$. By \cite[Remark 4.10]{alas2012} we can instead assume that $H_1$ is a locally convex space over $\R$ and $H_2,F$ are locally convex spaces over $\C$. 
 Then a map $f \colon U \rightarrow F$ is a called $C^{r,s}_{\R,\C}$-map if the iterated differentials $d^{(i,j)}f$ (as in (a)) exist for all $0 \leq i \leq r , 0 \leq j \leq s$ and are continuous.
 Note that here the derivatives in the first component are taken with respect to $\R$ and in the second component with respect to $\C$.
 \end{enumerate}
\end{defn}

\begin{rem} 
             One can extend the definition of $C^{r,s}$- and $C^{r,s}_{\R,\C}$-maps to obtain $C^{r,s}$- or $C^{r,s}_{\R,\C}$-mappings on a product $I \times V$, where $I$ is a closed interval and $V$ open.
  This works as in the case of $C^r$-maps (see Definition \ref{defn: non-open} or cf.\ \cite[Definition 3.2]{alas2012}). 
            For further results and details on the calculus of $C^{r,s}$-maps we refer to
\cite{alas2012}. 
\end{rem}

In the next proposition we will explicitly consider $C^0 ([0,1],\Lf(\BGC))$ as a locally convex vector space over $\C$.

\begin{prop}\label{prop: sm:flow}
The flow-map $\Fl^f  \colon [0,1] \times C^0 ([0,1],\Lf (\BGC)) \rightarrow \BGC$ is $C^{1,\infty}_{\R,\C}$.
\end{prop}

\begin{proof}
 Consider first a related finite-dimensional problem and define for $d \in \N$
  \begin{align*}
   G_d \colon [0,1] \times \C^{d} \times C^0 ([0,1], &\Lf (\BGC)) \rightarrow \C^{d} ,\\
   G_d (t,(x_1,\ldots,x_d),\an) &\coloneq (g_{1,\an} (t,x_1),g_{2,\an} (t,x_1,x_2), \ldots , g_{d,\an} (t,x_1,\ldots,x_d))
  \end{align*}
 with $g_{k,\an}$ as in Remark \ref{rem: ctbl:sys}.
 We claim that $G_d$ is of class $C^{0,\infty}_{\R,\C}$ with respect to the splitting $[0,1] \times (\C^{d} \times C^0 ([0,1], \Lf (\BGC))$ for all $d \in \N$.
 If this is true, the proof can be completed as follows. 
 Note that for fixed $\an \in C^0([0,1], \Lf (\BGC))$ and $d \in \N$ we obtain an inhomogeneous linear initial value problem 
  \begin{equation}\label{eq: ODE:findim}
   x'(t) = G_d (t,x(t),\an) , x(0) = 0 \in \C^d
  \end{equation}
 on the finite dimensional vector space $\C^d$. 
 Hence, for each fixed $\an$ there is a global solution $x^d_{0,\an} \colon [0,1] \rightarrow \C^d$ of \eqref{eq: ODE:findim}.
 Define the flow associated to \eqref{eq: ODE:findim} via 
  \begin{displaymath}
   \Fl^{G_d} \colon [0,1] \times C^{0} ([0,1], E) \rightarrow \C^d , (t,\an) \mapsto x^d_{0,\an} (t).
  \end{displaymath}
 As the right-hand side of \eqref{eq: ODE:findim} is a $C^{0,\infty}_{\R,\C}$-map, \cite[Proposition 5.9]{alas2012} shows that $\Fl^{G_d}$ is a mapping of class $C^{1,\infty}_{\R,\C}$.
 Define $\Pr_d \colon \C^\RT \rightarrow \C^d, \Pr_d \coloneq (\ev_{\tau_1},\ev_{\tau_2},\ldots,\ev_{\tau_d})$ and conclude from Remark \ref{rem: ctbl:sys} 
  \begin{displaymath}
   (\ev_{\tau_1} \circ \Fl^f, \ev_{\tau_2} \circ \Fl^f , \ldots , \ev_{\tau_d} \circ \Fl^f) = \PR_d \circ \Fl^f = \Fl^{G_d}.
  \end{displaymath}
 Now $\Fl^{G_d}$ is a $C^{1,\infty}_{\R,\C}$ mapping, whence the components $\ev_{\tau_l} \circ \Fl^f$ for $1\leq l \leq d$ are of class $C^{1,\infty}_{\R,\C}$. 
 As $d \in \N$ was arbitrary, all components of $\Fl^f$ are of class $C^{1,\infty}_{\R,\C}$. 
 The space $\C^\RT$ carries the product topology and thus \cite[Lemma 3.10]{alas2012} shows that $\Fl^f$ is a $C^{1,\infty}_{\R,\C}$-map as desired. 
 \bigskip
 
 \textbf{Proof of the claim, $G_d$ is $C^{0,\infty}_{\R,\C}$ for all $d \in \N$.}
 By \cite[Lemma 3.10]{alas2012}, $G_d$ will be of class $C^{0,\infty}_{\R,\C}$ if each of its components  is a $C^{0,\infty}_{\R,\C}$-map. 
 Thus if $\pi_k \colon \C^d \rightarrow \C$ is the projection onto the $k$th component, we have to prove that $\pi_k \circ G_d$ is a  $C^{0,\infty}_{\R,\C}$-map. 
 From \eqref{eq: d:trans} (cf.\ \eqref{eq: lower:diag}) we derive 
  \begin{displaymath}
   \pi_k \circ G_d (t,(x_1,\ldots,x_d),\an) = g_{k,\an} (t,(x_1,\ldots ,x_d)) = \ev_{\tau_k}(\an (t))  + \sum_{l< k} A_{kl} (t,\an) x_l.
  \end{displaymath}
 Recall that by \cite[Proposition 3.20]{alas2012} the evaluation map 
  \begin{displaymath}
   \ev \colon [0,1]\times C^0 ([0,1] , \Lf (\BGC)) \rightarrow \Lf (\BGC) , (t,\an) \mapsto \an (t)
  \end{displaymath}
 is a $C^{0,\infty}_{\R,\C}$-map.
 Hence $\ev_{\tau_k} (\an (t)) = \ev_{\tau_k} \circ \ev (t,\an)$ is a map of class $C^{0,\infty}_{\R,\C}$ by the chain rule \cite[Lemma 3.18]{alas2012}.
 
 Now consider the other summands. 
 The chain rules for $C^{r,s}_{\R,\C}$-mappings \cite[Lemma 3.17 and Lemma 3.19]{alas2012} 
 show that $A_{kl} (t,\an) \cdot x_l$ will be of class $C^{0,\infty}_{\R,\C}$ with respect to the splitting $[0,1] \times (\C^d \times C^0 ([0,1], \Lf (\BGC))$  
 if $A_{kl} \colon [0,1] \times C^0([0,1] , \Lf (\BGC)) \rightarrow \C$ is a $C^{0,\infty}_{\R,\C}$-map.
 Recall from the proof of Proposition \ref{prop: solution} (a) that each of the maps $A_{kl}$ is a finite 
 (possibly empty) sum of terms of the form $\an(t)(\tau_k \setminus s)$ with $s_{\tau_k} = \tau_l$.
 As above, these maps are a composition of the form $\ev_{\tau} \circ \ev$ whence of class $C^{0,\infty}_{\R,\C}$.
 We conclude that the maps $A_{kl}$ are of class $C^{0,\infty}_{\R,\C}$. 
 Summing up, $G_d$ is a $C^{0,\infty}_{\R, \C}$-mapping with respect to the splitting $[0,1] \times (\C^d \times C^0 ([0,1] , \Lf (\BGC)))$. 
\end{proof}

\begin{thm}
 \begin{enumerate}
  \item The complex Butcher group $\BGC$ is $C^0$-regular and its evolution  
  $$\evol_{\BGC} \colon C^0 ([0,1], \Lf(\BGC)) \rightarrow \BGC$$ is even holomorphic. 
  \item The Butcher group $\BGp$ is $C^0$-regular and its evolution map 
  $$\evol_{\BGp} \colon C^0 ([0,1], \Lf(\BGp)) \rightarrow \BGp$$ is even real analytic. 
 \end{enumerate}
In particular, both groups are regular in the sense of Milnor. \label{thm: BGp:reg}
\end{thm}

\begin{proof}

 \begin{enumerate}
  \item By Proposition \ref{prop: solution} the differential equation \eqref{eq: ODE:regular} admits a (unique) solution on $[0,1]$ whence we obtain the flow of \eqref{eq: ODE:regular}
   \begin{displaymath}
    \Fl^f \colon [0,1] \times C^0 ([0,1] , \Lf(\BGC)) \rightarrow \BGC.
   \end{displaymath}
   By Proposition \ref{prop: sm:flow}, $\Fl^f$ is a $C^{1,\infty}_{\R,\C}$-map.
 In particular, for $\an \in C^0 ([0,1], \Lf(\BGC))$ we obtain a $C^1$-curve $\Fl^f (\cdot , \an) \colon [0,1] \rightarrow \BGC$ which solves
 \eqref{eq: ODE:regular}, whence $\Fl^f (\cdot , \an)$ is the right product integral of the curve $\an$.
 Fixing the time, we obtain a smooth and even holomorphic mapping 
  \begin{displaymath}
   \evol \coloneq \Fl^f (1 , \cdot) \colon C^0 ([0,1] , \Lf(\BGC)) \rightarrow \BGC
  \end{displaymath}
  sending a curve $\an \in C^0([0,1], \Lf(\BGC))$ to the time $1$ evolution of its right product integral. 
  In summary, every continuous curve into $\Lf(\BGC)$ possesses a right product integral and the evolution map is smooth, 
  i.e.\ $\BGC$ is $C^0$-regular.
  Taking the complex structure into account, the evolution map is even holomorphic.
  \item Follows directly from part (a) and \cite[Corollary 9.10]{1208.0715v3} since $\BGC$ is a complexification of $\BGp$. 
 \qedhere
 \end{enumerate}
\end{proof}

\section{The Butcher group as an exponential Lie group}\label{sect: LGP:exp}

In the last section we have seen that the Butcher group (and the complex Butcher group) are $C^0$-regular. 
Restricting the evolution map to constant curves we obtain the exponential map. 
In the following, we identify $\Lf(\BGC)$ with the constant curves in $C^0([0,1], \Lf(\BGC))$ and write $\an$ for the constant curve $t\mapsto \an$.
Namely, we have

\begin{setup}
\label{setup: exponential}
For the complex Butcher group $\BGC$ the Lie group exponential map is given by 
\begin{displaymath}
\exp_{\BGC} \colon \Lf(\BGC) \to \BGC, \exp_{\BGC}(\an) = \evol_{\BGC} (t\mapsto \an) =\Fl^f(1, \an).
\end{displaymath}
By Theorem \ref{thm: BGp:reg} $\exp_{\BGC}$ is holomorphic and $\exp_{\BGp}$ is a real analytic map. 

To ease the computation, we choose and fix an enumeration of $\RT$ which satisfies \eqref{eq: ordering}. 
Now the curve $\gamma_{\an} (s) \coloneq \Fl^f(s, \an)$ is the solution to a countable system of differential equations.
Describing the system componentwise, we obtain for the $k$th component $\ev_{\tau_k} (\gamma_{\an}(t))$ of $\gamma_{\an}$ the differential equation
\begin{equation} \label{eq: ODE:exponential}
\ev_{\tau_k}(\gamma_{\an}'(t)) = \ev_{\tau_k} (\an)+ \sum_{l<k}A_{kl}(t,\an) \ev_{\tau_l}(\gamma_{\an} (t)), \qquad \ev_{\tau_k}(\gamma_{\an}(0))=0,
\end{equation}
where $A_{kl}(t ,\an)$ is a polynomial in $\{\ev_{\tau_1} (\an(t)), \dotsc,\ev_{\tau_{k-1}}(\an(t))\}$.
In this case, $A_{kl}(t, \an)$ is constant in $t$ (as $\an$ is constant in $t$) and we will write $A_{kl}  (\an) \coloneq A_{kl} (t, \an)$.
\end{setup}

We have already seen that $\exp_{\BGC}$ is a holomorphic and thus complex analytic mapping. 
Now we claim that $\exp_{\BGC}$ is a bijection whose inverse $\exp_{\BGC}^{-1}\colon \BGC \to L(\BGC)$ is complex analytic.

\begin{prop} \label{prop: logarithm}
For $b\in \BGC$, the equation $\exp_{\BGC}(\an)=b$ has exactly one solution $\LOG_{\BGC} (b) \in \Lf(\BGC)$.
If $b$ is contained in the subgroup $\BGp$, then $\LOG_{\BGC} (b)$ is contained in the real subalgebra $\Lf(\BGp)$.
\end{prop}

Note that an algebraic formula for $\LOG_{\BGC} (b)$ is derived in \cite[IX.9.1]{HLW2006} using similar methods as in the following proof.

\begin{proof}[Proof of \ref{prop: logarithm}]
We seek $\an \in \Lf(\BGC)$ such that the $C^1$-curve $\gamma_{\an} \colon [0,1] \rightarrow \BGC$ which solves \eqref{eq: ODE:exponential} for all $k\in \N$ also satisfies $\gamma_{\an}(1)=b$.
Thus we seek a curve which satisfies \eqref{eq: ODE:exponential} and $\ev_{\tau_k}(\gamma_{\an} (1)) = \ev_{\tau_k} (b)$, for all $k \in \N$.

Construct $\an$ by induction over $k \in \N$ (using the enumeration of trees).
Thus let $k=1$, i.e.\ $\tau_1 = \bullet$ (the one node tree) and consider \eqref{eq: ODE:exponential} and the above condition. 
We obtain 
\begin{equation}\label{eq: logarithmone}
 b(\bullet) = \ev_{\bullet}(\gamma_{\an}(1)) \text{ and }  \begin{cases}
                                               \ev_{\bullet}(\gamma_{\an}'(t)) &= \ev_\bullet (\an) = \an (\bullet) \\
                                               \ev_{\bullet}(\gamma_{\an}(0)) &=0,
                                              \end{cases}.
\end{equation}
Set $\ev_{\bullet}(\gamma_{\an} (t)) = t b(\bullet)$ for $0\leq t\leq 1$ to obtain a $C^1$-curve which satisfies \eqref{eq: logarithmone}. 
This entails $\ev_\bullet  (\an) = \an (\bullet) = b(\bullet)$.

Having dealt with the start of the induction, assume now that for $k > 1$ the values $\an (\tau_1), \dotsc, \an (\tau_{k-1})$ of $\an$ are known.
From the proof of Proposition \ref{prop: solution}, we then also know $\ev_{\tau_1}\circ\gamma_{\an}, \dotsc, \ev_{\tau_{k-1}}\circ\gamma_{\an}$.
Now $\ev_{\tau_k}\circ\gamma_{\an}$ is determined by the two conditions
\begin{equation}\label{eq: odecomponent}
\ev_{\tau_k}(\gamma_{\an} (1)) = b(\tau_k) \text{ and } \begin{cases}
                                               \ev_{\tau_k}(\gamma_{\an}'(t)) &= \an (\tau_k) + \sum_{l<k} A_{kl}(\an) \ev_{\tau_l}(\gamma_{\an}(t)), \\
                                               \ev_{\tau_k}(\gamma_{\an}(0)) &=0
                                              \end{cases}.
\end{equation}
The fundamental theorem of calculus \cite[Theorem 1.5]{hg2002a} allows us to rewrite the condition \eqref{eq: odecomponent} as 
\begin{equation}\label{eq: logimplicit} \begin{aligned}
                                       b(\tau_k) &= \ev_{\tau_k}(\gamma_{\an}(1)) = \ev_{\tau_k}(\gamma_{\an}(1)) - \ev_{\tau_k}(\gamma_{\an}(0)) = \int_0^1 \ev_{\tau_k}(\gamma_{\an} (t))\, \di t \\
						     &= \an(\tau_k) + \sum_{l<k} A_{kl}(\an) \int_{0}^{1} \ev_{\tau_l}(\gamma_{\an}(t))\, \di t.
                                        \end{aligned}
\end{equation}
Recall that the polynomials $A_{kl}(\an)$ depend only on the value of the first $k-1$ components of $\an$. 
As those together with $\ev_{\tau_l}\circ\gamma_{\an}$ for $1\leq l < k$ are known, this defines $\an (\tau_k)$.

If $b$ is contained in $\BGp$, then inductively, \eqref{eq: logarithmone} and \eqref{eq: logimplicit} show that $\an$ takes as values only real numbers and each of the $\ev_{\tau_k}\gamma_{\an}$ is real valued (cf.\ Proposition \ref{prop: solution}). 
We conclude that $\LOG_{\BGC} (b)$ is contained in $L(\BGp)$ if $b$ is in $\BGp$.
\end{proof}

\begin{prop}\label{prop: log:diff}
 With the notation of Proposition \ref{prop: logarithm} we define maps 
 \begin{align*}
  \exp_{\BGC}^{-1} &\colon \BGC \rightarrow \Lf(\BGC), b \mapsto \LOG_{\BGC} (b) \\
  \exp_{\BGp}^{-1} &\colon \BGp \rightarrow \Lf(\BGp), b \mapsto \LOG_{\BGC} (b).
 \end{align*}
Then $\exp_{\BGC}^{-1}$ is holomorphic and $\exp_{\BGp}^{-1}$ is real analytic.
\end{prop}
\begin{proof}
 Since $\BGC$ is a complexification of $\BGp$, Proposition \ref{prop: logarithm} shows that it suffices to prove the assertions for $\exp_{\BGC}^{-1}$.
 We define for each $k \in \N$ a map 
  \begin{displaymath}
   \LOG_k \colon \C^k \to \C, \LOG_k ((z_1,\ldots, z_k)) \coloneq \ev_{\tau_k} (\exp_{\BGC}^{-1} (\hat{z})),
  \end{displaymath}
 where $\hat{z} \in \BGC$ with $\ev_{\tau_i} (\hat{z}) = z_i$ for all $1\leq i \leq k$.
  Recall from \eqref{eq: logimplicit} that the value of $\ev_{\tau_k}\circ \exp_{\BGC}^{-1}(b)$ only depends on $\PR_k(b)= (\ev_{\tau_1} (b), \ldots, \ev_{\tau_k} (b))$, whence $ \LOG_k$ is well-defined.
 Furthermore, $\LOG_k \circ \PR_k = \ev_{\tau_k} \circ \exp_{\BGC}^{-1}$, where $\PR_k\colon \C^{\RT}\to \C^k, \PR_k \coloneq (\ev_{\tau_1}, \ev_{\tau_2}, \dotsc, \ev_{\tau_k})$ is holomorphic.
 Since a map into a product is holomorphic if its components are holomorphic, $ \exp_{\BGC}^{-1}$ is holomorphic if for each $k \in \N$ the map $\LOG_k$ is holomorphic.
 We proceed by induction on $k \in \N$.\bigskip
 
 For $k=1$ the identity \eqref{eq: logarithmone} shows that $\LOG_1$ is $\id_\C$ and thus holomorphic.

 Now let $k> 1$ and assume that for $l < k$ the mappings $\LOG_l \colon \C^l \rightarrow \C$ are holomorphic.
 We show that $\LOG_k$, implicitly defined by \eqref{eq: logimplicit}, splits into a composition of holomorphic mappings.
For $\an \in \Lf (\BGC)$ let $\gamma_{\an}\colon [0,1] \to \BGC$ be the curve $\gamma_{\an} (t) = \exp_{\BGC} (t\an)$. 
As $\ev_{\tau_l} \circ \gamma_ {\an}$ depends only on $(z_1,\ldots,z_l)\in \C^l$ and solves \eqref{eq: ODE:exponential} for all $l<k$, we derive 
  \begin{align*}
   \LOG_k (z_1,\ldots , z_k) &= z_k - \sum_{l < k} A_{kl} (\hat{\cn}) \int_0^1 \gamma_{l,\hat{\cn}} (t) \di t.\\
   \text{with } \hat{\cn} \in \Lf (\BGC)  \text{ such that } \ev_{\tau_l} (\hat{\cn}) &= \LOG_l (z_1,\ldots, z_l) \text{ for all } 1\leq l <k.
   \end{align*}
 Here $A_{kl}$ is a polynomial in the first $l$ components of $\hat{\cn}$ whence the previous formula is well defined.
 Consider  
 \begin{displaymath} 
  L \colon \C^k \rightarrow \C^{k-1}, (z_1, \dotsc, z_k) \mapsto (\LOG_1 (z_1) , \dotsc, \LOG_{k-1} (z_1, \dotsc , z_{k-1})),
 \end{displaymath}
 By the induction hypothesis, $\LOG_l$ is holomorphic for $l<k$, so $L$ is holomorphic.
 For $1\leq l< k$ the map $\gamma_l \colon \C^{l} \rightarrow C^1 ([0,1] , \C)$ sending $\Pr_k (\hat{\cn})$ to the solution $\ev_{\tau_l}\circ \gamma_{\hat{\cn}}$ of \eqref{eq: ODE:exponential} is holomorphic.
 Here $C^1 ([0,1], \C)$ with the topology of uniform convergence is a complex Banach space (cf.\ Example \ref{ex: funspace}).
 The exponential law \cite[Theorem A]{alas2012} implies that $\gamma_l$ will be holomorphic if and only if
 $\gamma_l^\vee \colon [0,1] \times \C^l \rightarrow \C, \gamma_l^\vee (t,z)\coloneq \gamma_l(z)(t)$ is a $C^{1,\infty}_{\R,\C}$-map.
 By construction $\gamma_l^\vee$ is the flow associated to the differential equation \eqref{eq: ODE:exponential}. 
 Note that by the induction hypothesis the right hand side of the differential equation is a $C^{0,\infty}_{\R,\C}$-mapping. 
 From \cite[Theorem C]{alas2012} we infer that $\gamma_l^\vee$ is a $C^{1,\infty}_{\R,\C}$-mapping and thus $\gamma_l$ is holomorphic.
 Define for $1\leq l < k$ the map 
 \begin{align*}
  \Gamma_l\colon \C^{k-1} &\rightarrow \C, (z_1,\ldots, z_{k-1}) \mapsto \int_0^1 \gamma_l (z_1,\ldots, z_l)(t) \di t .
 \end{align*}
 Recall that the integral operator $\int_0^1 \colon C^1 ([0,1] , \C) \rightarrow \C$ is continuous linear and $\gamma_l$ is holomorphic. 
 Thus $\Gamma_l$ is holomorphic.  
 Finally, write  
 \begin{align*}
  \LOG_k (z_1,\ldots, z_k) = z_k + \sum_{1 \leq l<k} A_{kl} \circ L (z_1,\ldots , z_{k-1}) \cdot \Gamma_l \circ L (z_1,\ldots z_{k-1})
 \end{align*}
 as a composition of holomorphic maps, whence $\LOG_k$ is holomorphic. 
\end{proof}

From Proposition \ref{prop: log:diff} we immediately deduce the following theorem.

\begin{thm} \label{thm: exp:ana}
Let $G$ be either the complex Butcher group $\BGC$ or the Butcher group $\BGp$. 
Then $G$ is an exponential Lie group i.e.\ the exponential map $\exp_G \colon \Lf(G) \rightarrow G$ is a global diffeomorphism.
Note that $\exp_G$ is even an analytic diffeomorphism
\end{thm}

\begin{rem}
 Recall from \cite[Definition IV.1.9]{neeb2006} that a Lie group whose associated exponential map is an analytic (local) diffeomorphism is a \emph{Baker--Campbell--Hausdorff (BCH)} Lie group.
Thus Theorem \ref{thm: exp:ana} shows that $\BGC$ and $\BGp$ are BCH Lie groups. 

Note that this entails that $\Lf(\BGC)$ and $\Lf (\BGp)$ are BCH Lie algebras, 
i.e.\ these Lie algebras admit a zero-neighbourhood $U$ such that for all $x,y \in U$ the Baker--Campbell--Hausdorff-series $\sum_{n=1}^\infty H_n (x,y)$ converges (see \cite[Definition IV.1.3]{neeb2006}) and defines an analytic product.
In fact  from \cite[Theorem IV.2.8]{neeb2006} we derive that the BCH-series is the Taylor series of the local multiplication (cf.\ Theorem \ref{thm: Butcher:Lie} and Section \ref{sect: BG-LAlg})
  \begin{displaymath}
   \ast \colon \Lf (\BGC) \times \Lf (\BGC) \rightarrow \Lf (\BGC), \an\ast \bn = (\an +e) \cdot (\bn+e) -e.
  \end{displaymath}
\end{rem}

\section{The subgroup of symplectic tree maps}\label{sect: subgp}

In this section we show that the subgroup of symplectic tree maps is a closed Lie subgroup of $\BGC$.

\begin{rem}
 Recall from (cf.\ \cite[p.81]{Butcher72}) the definition of the \emph{Butcher product} (not to be confused with the product in the Butcher group): 
 
 For two trees $u,v$ we denote by $u \circ v$ the Butcher product, 
 defined as the rooted tree obtained by adding an edge from the root of $v$ to the root of $u$, 
 and letting the root of $u$ be the root of the full tree.
 
 Let us illustrate the Butcher product with some examples involving trees with one and two nodes (in the picture the node at the bottom is the root of the tree):

  \begin{displaymath}
      \begin{tikzpicture}[dtree, scale=1.5]
        \node[dtree black node] {}
        ;
      \end{tikzpicture}
       \circ 
          \begin{tikzpicture}[dtree,scale=1.5]
            \node[dtree black node] {}
            ;
          \end{tikzpicture} 
          =  
          \begin{tikzpicture}[dtree,scale=1.5]
             \node[dtree black node] {}
             child { node[dtree black node] {} }
             ;
           \end{tikzpicture}, \quad \quad
             \begin{tikzpicture}[dtree]
               \node[dtree black node] {}
               ;
             \end{tikzpicture} \circ
              \begin{tikzpicture}[dtree, scale=1.5]
                \node[dtree black node] {}
                child { node[dtree black node] {} }
                ;
              \end{tikzpicture} = 
               \begin{tikzpicture}[dtree, scale=1.5]
                 \node[dtree black node] {}
                 child { node[dtree black node] {}
                 child {node[dtree black node] {}}}
                 ;
               \end{tikzpicture}
               , \quad \quad
               \begin{tikzpicture}[dtree, scale=1.5]
                 \node[dtree black node] {}
                 child { node[dtree black node] {} }
                 ;
               \end{tikzpicture}
               \circ
                         \begin{tikzpicture}[dtree, scale=1.5]
                            \node[dtree black node] {}
                            ;
                          \end{tikzpicture} =
                            \begin{tikzpicture}[dtree, scale=1.5]
                              \node[dtree black node] {}
                              child { node[dtree black node] {}
                              }
                              child {node[dtree black node] {}}
                              ;
                            \end{tikzpicture}, \quad \quad
                \begin{tikzpicture}[dtree, scale=1.5]
                  \node[dtree black node] {}
                  child { node[dtree black node] {} }
                  ;
                \end{tikzpicture} \circ
                  \begin{tikzpicture}[dtree, scale=1.5]
                    \node[dtree black node] {}
                    child { node[dtree black node] {} }
                    ;
                  \end{tikzpicture}=
                   \begin{tikzpicture}[dtree, scale=1.5]
                     \node[dtree black node] {}
                     child { node[dtree black node] {}
                     }
                     child {node[dtree black node] {}
                     child { node[dtree black node] {}
                     }}
                     ;
                   \end{tikzpicture}
  \end{displaymath}
\end{rem}

\begin{defn}\begin{enumerate}
             \item A tree map $a$ is called \emph{symplectic} if it satisfies the condition  
\begin{equation}\label{eq: symplectic}
 P_{u,v}(a) \coloneq a(u \circ v) + a(v \circ u) - a(u) a(v)=0, \qquad \forall u, v \in \RT.
\end{equation}
\item We let $\SymplTM$ be the \emph{subset of all symplectic tree maps} in $\BGC$. 
Note that the group multiplication of the Butcher group turns $\SymplTM$ into a subgroup (see Lemma \ref{lem: ssgroup:clsd}). 
\end{enumerate}\label{defn: sympl:tree}
\end{defn}

For the reader's convenience we recall the proof of the subgroup property for $\SymplTM$. 

\begin{lem}\label{lem: ssgroup:clsd}
 The set $\SymplTM$ of symplectic tree maps is a closed subgroup of $\BGC$.
\end{lem}

\begin{proof} Let us first establish the subgroup property.
To see that $\SymplTM$ is a subgroup one uses \cite[Theorem VI.7.6]{HLW2006} twice. 
 First by \cite[Theorem VI.7.6]{HLW2006} the Butcher series associated to a symplectic tree map preserves certain quadratic first integrals. 
 Now the product in the Butcher group corresponds to the composition of Butcher series (cf.\ \cite[III.1.4]{HLW2006}). 
 As the Butcher series preserve the quadratic first integrals the same holds for their composition, whence by  \cite[Theorem VI.7.6]{HLW2006} the product of symplectic tree maps is a symplectic tree map.
 
 To see that $\SymplTM$ is closed, note that $P_{u,v}$ is continuous for all $u,v \in \RT$.
 In fact this follows from the continuity of $\ev_{u\circ v}, \ev_{v\circ u}, \ev_{u}$ and $\ev_{v}$ (see Lemma \ref{lem: bgp:prod}).
 Now $\SymplTM = \bigcap_{u,v \in \RT} P_{u,v}^{-1} (0)$ is closed as an intersection of closed sets.  
 \end{proof} 

\begin{rem} Recall from \cite[Theorem VI.7.6]{HLW2006} that the integration method associated to a symplectic tree map is symplectic for general Hamiltonian systems $y' = J^{-1} \nabla H(y)$.
 Thus it becomes clear why the tree maps which satisfy \eqref{eq: symplectic} are called symplectic.
\end{rem}

\begin{setup}\label{setup: diff:tang}
 Recall from \cite[Proposition II.6.3]{neeb2006} that one can associate to the subgroup $\SymplTM \subseteq \BGC$ the \emph{differential tangent set} 
  \begin{displaymath}
   \Lf^d (\SymplTM) \coloneq \{\alpha'(0) \in T_e \BGC \mid \alpha \in C^1 ([0,1], \BGC), a(0) = e \text{ and } a([0,1]) \subseteq \SymplTM\}
  \end{displaymath}
 which is a Lie subalgebra of $\Lf (\BGC) = T_e \BGC$.
 Again we identify in the following the tangent space $T_a \BGC$ with $\C^{\RT}$. 
 \medskip
 
 Let us compute $\Lf^d (\SymplTM)$. 
 Consider $\gamma \in C^1 ([0,1], \BGC), \gamma (0) = e \text{ and } \gamma ([0,1]) \subseteq \SymplTM$.
 Observe that for a tree $u$, the map $\ev_u \circ \gamma \colon [0,1] \rightarrow \C$ is smooth and the chain rule yields $\tfrac{\partial}{\partial t} \ev_u \circ \gamma(t) = \ev_u (\tfrac{\partial}{\partial t} \gamma (t))$.
 Thus, for all $u,v \in \RT$ we have:
 \begin{equation} \label{eq: Pdiff} \begin{aligned}
\frac{\partial}{\partial t} P_{u, v}(\gamma(t)) =  &\left( \frac{\partial}{\partial t} \gamma(t)\right)(u \circ v) + \left( \frac{\partial}{\partial t}\gamma(t)\right)(v \circ u) -\\
 &\gamma(t)(v)  \left( \frac{\partial}{\partial t}\gamma(t)\right)(u) -  \gamma(t)(u)  \left( \frac{\partial}{\partial t}\gamma(t)\right)(v)=0.
 \end{aligned}
\end{equation}
 In particular for $t= 0$ this entails $\left(\left.\tfrac{\partial}{\partial t}\right|_{t=0}\gamma(t)\right)(u \circ v) + \left(\left.\tfrac{\partial}{\partial t}\right|_{t=0}\gamma(t)\right)(v \circ u) = 0$.
 The differential tangent set for the subgroup $\SymplTM$ is therefore given by 
 \begin{equation}\label{eq: symplecticalg}
  \Lf^d (\SymplTM) = \{\bn\in \Lf(\BGC) \mid Q_{u,v} (\bn) \coloneq \bn(u\circ v)+\bn(v\circ u)= 0\, \text{ for all } u,v\in \RT\}.
  \end{equation}
 Since $\SymplTM$ is a closed subgroup of a (locally) exponential Lie group, \cite[Lemma IV.3.1.]{neeb2006} shows that $\Lf^d (\SymplTM)$ is a closed Lie subalgebra and we can identify it with 
  \begin{displaymath}
   \Lf^d (\SymplTM) = \{x \in \Lf (\BGC) \mid \exp_{\BGC} (\R x) \subseteq \SymplTM\}.
  \end{displaymath}
\end{setup}

\begin{rem}
 The characterisation \eqref{eq: symplecticalg} of the differential tangent set of $\SymplTM$ exactly reproduces the condition in \cite[Remark IX.9.4]{HLW2006}.
 There the condition \eqref{eq: symplecticalg} characterises an element ``in the tangent space at the identity of $\SymplTM$''. 
 Note that in loc.cit.\ no differentiable structure on $\BGC$ or $\SymplTM$ is considered and a priori it is not clear whether $\SymplTM$ is actually a submanifold of $\BGC$.
 The differentiable structure of the Butcher group allows us to exactly recover the intuition of numerical analysts. 
 Indeed we will see that $\SymplTM$ is a submanifold of $\BGC$ such that $T_e \SymplTM = \Lf^d (\SymplTM)$. 
\end{rem}

\begin{prop}\label{prop: ssgp:alg}
 Let $\SymplTM$ be the subgroup of symplectic tree maps, then 
  \begin{displaymath}
   \exp_{\BGC} (\Lf^d (\SymplTM)) = \exp_{\BGC} (\Lf (\BGC)) \cap \SymplTM.
  \end{displaymath}
\end{prop}

\begin{proof}
 The characterisation of $\Lf^d (\SymplTM)$ in \ref{setup: diff:tang} shows that $\exp_{\BGC} (\Lf^d (\SymplTM)) \subseteq \SymplTM$.
 
 For $\SymplTM\subseteq \exp_{\BGC} (\Lf^d (\SymplTM))$, let $a= \exp_{\BGC}(\bn) \in \SymplTM$ and recall that $a$ uniquely determines $\bn \in \Lf(\BGC)$ since $\BGC$ is exponential.
 Also define $\gamma\colon [0,1] \to \BGC$ as $\gamma(t) = \exp_{\BGC}(t\bn)$.

 We must show that $\bn$ satisfies \eqref{eq: symplecticalg} for all pairs of trees $u,v$.
 As a side product of the proof, we get that $P_{u,v}(\gamma(t))=0$ for all $t$.
 Proceed by induction on $|u|+|v|$, and consider \eqref{eq: Pdiff} which we evaluate for all trees using that $\gamma$ solves the differential equation \eqref{eq: ODE:regular} for the constant path $t \mapsto \bn$, i.e.\ since $\bn (\emptyset)=0$ we have
 \begin{equation} \label{eq: gammadiff}
  \left(\frac{\partial}{\partial t}\gamma(t)\right)(\tau) =  \bn(\tau_k) + \sum_{s\in \SP{\tau_k}_1} \gamma (t)(s_{\tau_k}) \bn(\tau_k \setminus s).
 \end{equation}

First, let $|u|+|v|=2$, i.e.\ we insert the single node tree $u=v=\onenode$ in \eqref{eq: symplectic}.
Now $P_{\onenode, \onenode}(a) = 2a(\twonode)-a(\onenode)^ 2$ and \eqref{eq: symplecticalg} yields $Q_{\onenode, \onenode}(\bn) = 2\bn(\twonode)$. 
Moreover, for the single node trees \eqref{eq: Pdiff} with \eqref{eq: gammadiff} yields
 \begin{displaymath}
  \frac{\partial}{\partial t} P_{\onenode, \onenode}(\gamma(t)) = 2\bn(\twonode) +2\gamma(t)(\onenode)\bn(\onenode) -2\gamma(t)(\onenode)\bn(\onenode)= 2\bn(\twonode).
 \end{displaymath}
 We conclude that $\frac{\partial}{\partial t} P_{\onenode, \onenode}(\gamma(t))$ is constant in $t$.
Now $\gamma(0)=e$ implies $P_{\onenode, \onenode}(\gamma(0))=0$  and $P_{\onenode, \onenode}(\gamma(1))=0$ holds since $\gamma(1) \in \SymplTM$.
Therefore, the fundamental theorem of calculus \cite[Theorem 1.5]{hg2002a} yields $Q_{\onenode, \onenode} (\bn) = 2\bn(\twonode)= \int_0^1 \frac{\partial}{\partial t} P_{\onenode, \onenode}(\gamma(t)) \dd t = 0$.
In addition, $P_{\onenode, \onenode}(\gamma(t))=0$ for all $t$.\bigskip

Now, let $u$, $v$ be arbitrary, and assume that $Q_{u',v'}(\bn)=0$ and $P_{u',v'}(\gamma(t))=0$ for all pairs of trees $u',v'$ where $|u'|+|v'|<|u|+|v|$.
For the first term in \eqref{eq: Pdiff}, we have
\begin{equation}\label{eq: first}
\left(\frac{\partial}{\partial t}\gamma(t)\right)(u \circ v) =  \bn(u \circ v) + \sum_{s\in \SP{u \circ v}_1} \gamma (t)(s_{u \circ v}) \bn(u \circ v \setminus s).
\end{equation}
Observe that the splittings $\SP{u \circ v}_1$ can be divided into three parts.
Namely, we have three disjoint cases for $s\in \SP{u \circ v}_1$, either $s_{u\circ v}=u$ or $s_{u\circ v}= (s_1)_{u} \circ v$ where $s_1\in \SP{u}_1$ or $s_{u\circ v}=u \circ (s_2)_v$ where $s_2\in \SP{v}_1$.
Therefore we can rewrite \eqref{eq: first} as
\begin{multline*}
\left(\frac{\partial}{\partial t}\gamma(t)\right)(u \circ v) =  \bn(u \circ v) + \gamma(t)(u) \bn(v) + \\
\sum_{s_1 \in \SP{u}_1} \gamma (t)((s_1)_u\circ v) \bn (u\setminus s_1)+  \sum_{s_2 \in \SP{v}_1} \gamma (t)(u \circ (s_2)_v) \bn(v\setminus s_2).
\end{multline*}
For the second term in \eqref{eq: Pdiff}, we get the same expression with $u$ and $v$ interchanged, while for the last two terms, we can use \eqref{eq: gammadiff} directly.

Now, we see that $\frac{\partial}{\partial t} P_{u,v}(\gamma(t))$ can be written as
\begin{multline*}
 \frac{\partial}{\partial t} P_{u, v}(\gamma(t)) = Q_{u,v}(\bn) +\\ 
  \sum_{s_1 \in \SP{u}_1} P_{(s_1)_u, v}(\gamma(t))\bn(u\setminus s_1) + \sum_{s_2 \in \SP{v}_1} P_{u, (s_2)_v}(\gamma(t)) \bn(v\setminus s_2).
\end{multline*}
By the induction hypothesis (since $|s_1|<|u|$ and $|s_2|<v$) the two sums disappear, and we are left with  $\frac{\partial}{\partial t} P_{u, v}(\gamma(t)) = Q_{u,v}(\bn)$.
Arguing as in the case $|u|+|v|=2$, we derive $Q_{u,v}(\bn)=0$ and therefore also $P_{u, v}(\gamma(t))=0$ for all $t$. Thus $\bn \in \Lf^d (\SymplTM)$ 
\end{proof}

\begin{thm}\label{thm: subL:sympl}
 The subgroup $\SymplTM$ is a closed Lie subgroup of $\BGC$. Its Lie algebra $\Lf (\SymplTM)$ coincides with $\Lf^d (\SymplTM)$.
 Moreover, this structure turns $\SymplTM$ into an exponential BCH Lie group.
\end{thm}

\begin{proof}
 Proposition \ref{prop: ssgp:alg} shows that \cite[Theorem IV.3.3.]{neeb2006} is applicable. 
 Hence the subspace topology turns $\SymplTM$ into a locally exponential Lie subgroup of $\BGC$.
 However, since $\BGC$ is exponential Proposition \ref{prop: ssgp:alg} indeed shows that $\SymplTM$ is an exponential Lie group.
 Moreover, the exponential map and its inverse are analytic mappings, whence by \cite[Definition IV.1.9]{neeb2006} the group $\SymplTM$ becomes a BCH Lie group.
\end{proof}

\begin{cor}
 The subgroup of real symplectic tree maps $\SymplTMR \coloneq \SymplTM \cap \BGp$ is a closed Lie subgroup of $\BGp$. 
 Its Lie algebra $\Lf (\SymplTMR)$ coincides with $\Lf^d (\SymplTM) \cap \Lf (\BGp)$.
 Moreover, this structure turns $\SymplTMR$ into an exponential BCH Lie group.
\end{cor}

 \section*{Acknowledgements}
The research on this paper was partially supported by the projects \emph{Topology in Norway} (Norwegian Research Council project 213458) and \emph{Structure Preserving Integrators, Discrete Integrable Systems and Algebraic Combinatorics} (Norwegian Research Council project 231632). The second author would also like to thank Reiner Hermann for helpful discussions on Hopf algebras.

\begin{appendix}
\phantomsection
\addcontentsline{toc}{section}{References}
\bibliographystyle{new}
\bibliography{butcher}

\newcommand{\noopsort}[1]{} \newcommand{\singleletter}[1]{\#1}
\begin{thebibliography}{MMMKV14}
\providecommand{\url}[1]{\texttt{#1}}
\providecommand{\urlprefix}{URL }
\expandafter\ifx\csname urlstyle\endcsname\relax
  \providecommand{\doi}[1]{doi:\discretionary{}{}{}#1}\else
  \providecommand{\doi}{doi:\discretionary{}{}{}\begingroup
  \urlstyle{rm}\Url}\fi
\providecommand{\eprint}[2][]{\url{#2}}

\bibitem[AS15]{alas2012}
Alzaareer, H. and Schmeding, A.
\newblock \emph{{Differentiable mappings on products with different degrees of
  differentiability in the two factors}}.
\newblock Expositiones Mathematicae  (2015)(33):184--222.
\newblock DOI:10.1016/j.exmath.2014.07.002

\bibitem[Bas64]{bastiani64}
Bastiani, A.
\newblock \emph{Applications diff\'erentiables et vari\'et\'es
  diff\'erentiables de dimension infinie}.
\newblock J. Analyse Math. \textbf{13} (1964):1--114

\bibitem[BDS15]{BDS2015}
Bogfjellmo, G., Dahmen, R. and Schmeding, A.
\newblock \emph{{Character groups of Hopf algebras as infinite-dimensional Lie
  groups}} 2015.
\newblock \urlprefix\url{http://arxiv.org/abs/1501.05221}

\bibitem[Bro04]{Brouder-04-BIT}
Brouder, C.
\newblock \emph{{Trees, renormalization and differential equations}}.
\newblock BIT Num. Anal. \textbf{44} (2004):425--438

\bibitem[But72]{Butcher72}
Butcher, J.~C.
\newblock \emph{{An algebraic theory of integration methods}}.
\newblock Math. Comp. \textbf{26} (1972):79--106

\bibitem[CHV10]{CHV2010}
Chartier, P., Hairer, E. and Vilmart, G.
\newblock \emph{{Algebraic Structures of B-series}}.
\newblock Foundations of Computational Mathematics \textbf{10}
  (2010)(4):407--427.
\newblock \doi{10.1007/s10208-010-9065-1}

\bibitem[CK98]{CK98}
Connes, A. and Kreimer, D.
\newblock \emph{{Hopf Algebras, Renormalization and Noncommutative Geometry}}.
\newblock Commun.Math.Phys. 199 203-242  (1998).
\newblock \doi{10.1007/s002200050499}

\bibitem[CMSS94]{CMSS93}
Calvo, M.~P., Murua, A. and Sanz-Serna, J.~M.
\newblock \emph{{Modified equations for {ODE}s}}.
\newblock In \emph{{Chaotic numerics ({G}eelong, 1993)}}, \emph{{Contemp.
  Math.}}, vol. 172, pp. 63--74 (Amer. Math. Soc., Providence, RI, 1994).
\newblock \doi{10.1090/conm/172/01798}.
\newblock \urlprefix\url{http://dx.doi.org/10.1090/conm/172/01798}

\bibitem[Dah11]{dahmen2011}
Dahmen, R.
\newblock \emph{{{D}irect {L}imit {C}onstructions in {I}nfinite {D}imensional
  {L}ie {T}heory}}.
\newblock Ph.D. thesis, {University of Paderborn} 2011.
\newblock \eprint{urn:nbn:de:hbz:466:2-239}

\bibitem[Dei77]{deimling77}
Deimling, K.
\newblock \emph{{Ordinary Differential Equations in Banach Spaces}}.
\newblock No. 596 in {Lecture Notes in Mathematics} (Springer Verlag,
  Heidelberg, 1977)

\bibitem[EFGBP07]{LT07}
Ebrahimi-Fard, K., Gracia-Bondia, J.~M. and Patras, F.
\newblock \emph{{A Lie theoretic approach to renormalization}}.
\newblock Commun.Math.Phys. 276 519-549  (2007).
\newblock \doi{10.1007/s00220-007-0346-8}

\bibitem[Gl{\"o}02a]{MR1934608}
Gl{\"o}ckner, H.
\newblock \emph{{Lie group structures on quotient groups and universal
  complexifications for infinite-dimensional {L}ie groups}}.
\newblock J. Funct. Anal. \textbf{194} (2002)(2):347--409.
\newblock \doi{10.1006/jfan.2002.3942}

\bibitem[Gl{\"o}02b]{hg2002a}
Gl{\"o}ckner, H.
\newblock \emph{{Infinite-dimensional Lie groups without completeness
  restrictions}}.
\newblock In A.~Strasburger, J.~Hilgert, K.~Neeb and W.~Wojty{\'n}ski (Eds.),
  \emph{{Geometry and Analysis on Lie Groups}}, \emph{{Banach Center
  Publication}}, vol.~55, pp. 43--59 (Warsaw, \noopsort{a}2002)

\bibitem[Gl{\"o}07]{MR2402519}
Gl{\"o}ckner, H.
\newblock \emph{Instructive examples of smooth, complex differentiable and
  complex analytic mappings into locally convex spaces}.
\newblock J. Math. Kyoto Univ. \textbf{47} (2007)(3):631--642

\bibitem[Gl{\"o}15]{1208.0715v3}
Gl{\"o}ckner, H.
\newblock \emph{{Regularity properties of infinite-dimensional Lie groups, and
  semiregularity}} 2015.
\newblock \urlprefix\url{http://arxiv.org/abs/1208.0715v3}.
\newblock \eprint{1208.0715v3}

\bibitem[Gou61]{MR0130553}
Gould, G.~G.
\newblock \emph{Locally unbounded topological fields and box topologies on
  products of vector spaces}.
\newblock J. London Math. Soc. \textbf{36} (1961):273--281

\bibitem[HL97]{HL1997}
Hairer, E. and Lubich, C.
\newblock \emph{{The life-span of backward error analysis for numerical
  integrators}}.
\newblock Numerische Mathematik \textbf{76} (1997)(4):441--462.
\newblock \doi{10.1007/s002110050271}.
\newblock \urlprefix\url{http://dx.doi.org/10.1007/s002110050271}

\bibitem[HLW06]{HLW2006}
Hairer, E., Lubich, C. and Wanner, G.
\newblock \emph{{Geometric Numerical Integration}}, \emph{{Springer Series in
  Computational Mathematics}}, vol.~31 (Springer Verlag, $^2$2006)

\bibitem[Jar81]{jarchow1980}
Jarchow, H.
\newblock \emph{{Locally Convex Spaces}}.
\newblock Lecture Notes in Mathematics 417 (Teubner, Stuttgart, 1981)

\bibitem[Kel74]{keller1974}
Keller, H.
\newblock \emph{{Differential Calculus in Locally Convex Spaces}}.
\newblock {Lecture Notes in Mathematics 417} (Springer Verlag, Berlin, 1974)

\bibitem[KM97]{KM97}
Kriegl, A. and Michor, P.~W.
\newblock \emph{{The convenient setting of global analysis}},
  \emph{{Mathematical Surveys and Monographs}}, vol.~53 (AMS, 1997)

\bibitem[Kni64]{MR0160184}
Knight, C.~J.
\newblock \emph{Box topologies}.
\newblock Quart. J. Math. Oxford Ser. (2) \textbf{15} (1964):41--54

\bibitem[Men05]{mencattini}
Mencattini, I.
\newblock \emph{{Structure of the insertion elimination {L}ie algebra in the
  ladder case}} (ProQuest LLC, Ann Arbor, MI, 2005).
\newblock Thesis (Ph.D.)--Boston University

\bibitem[Mic80]{michor1980}
Michor, P.
\newblock \emph{{Manifolds of Differentiable Mappings}}.
\newblock {Shiva Mathematics Series 3} (Shiva Publishing Ltd., Orpington, 1980)

\bibitem[Mil83]{milnor1983}
Milnor, J.
\newblock \emph{{Remarks on infinite-dimensional Lie groups}}.
\newblock In B.~DeWitt and R.~Stora (Eds.), \emph{{Relativity, Groups and
  Topology II}}, pp. 1007--1057 (North Holland, New York, 1983)

\bibitem[MMMKV14]{MMMV14}
McLachlan, R.~I., Modin, K., Munthe-Kaas, H. and Verdier, O.
\newblock \emph{{B-series methods are exactly the local, affine equivariant
  methods}} 2014.
\newblock \urlprefix\url{http://arxiv.org/abs/1409.1019}

\bibitem[Nee06]{neeb2006}
Neeb, K.
\newblock \emph{{Towards a Lie theory of locally convex groups}}.
\newblock Japanese Journal of Mathematics \textbf{1} (2006)(2):291--468

\bibitem[Sch71]{MR0342978}
Schaefer, H.~H.
\newblock \emph{{Topological vector spaces}} (Springer-Verlag, New York-Berlin,
  1971).
\newblock Graduate Texts in Mathematics, Vol. 3

\end{thebibliography}
\end{appendix}
\end{document}